\documentclass[a4paper,10pt]{article}
\usepackage[utf8]{inputenc}
\usepackage{graphicx}
\usepackage{amsmath} 
\usepackage{amssymb}
\usepackage{amsthm}
\usepackage[pdftex,colorlinks=true,urlcolor=blue,citecolor=black,anchorcolor=black,linkcolor=black]{hyperref}
\usepackage[a4paper,top=3.0cm, bottom=3.0cm, left=2.5cm, right=2.5cm]{geometry}
 \usepackage{array}
 \allowdisplaybreaks

\newtheorem{theorem}{Theorem}

\newtheorem{definition}{Definition}
\newtheorem{example}{Example}
\newtheorem{remark}{Remark}
\newtheorem{proposition}{Proposition}
\newtheorem{lemma}{Lemma}

 \numberwithin{theorem}{section}
 \numberwithin{lemma}{section}
 \numberwithin{example}{section}
 \numberwithin{remark}{section}
 \numberwithin{proposition}{section}
 \numberwithin{definition}{section}
 \numberwithin{equation}{section}

\newcolumntype{C}[1]{>{\centering\let\newline\\\arraybackslash\hspace{0pt}}m{#1}}
\newcommand{\integers}{{\mathbb Z}}
\newcommand{\reals}{{\mathbb R}}
\title{Regenerative Simulation for Queueing Networks with Exponential or Heavier Tail Arrival Distributions}
\author{Sarat Babu Moka \\ Sandeep Juneja \\ [8pt] 
School of Technology and Computer Science \\ 
Tata Institute of Fundamental Research \\
Mumbai - 400005, INDIA}

\begin{document}

\maketitle
\vspace{-4mm}
\begin{abstract}
Multiclass open queueing networks find wide applications in communication, computer and fabrication networks. Often one is interested in steady-state performance
measures associated with these networks. Conceptually, under mild conditions, a regenerative structure exists in multiclass networks, making them amenable to 
regenerative simulation for estimating the steady-state performance measures. However, typically, identification of a regenerative structure
in these networks is difficult. A~well known exception is when  all the interarrival times are exponentially distributed, where the instants corresponding to customer arrivals to 
an empty network constitute a regenerative structure. In this paper, we consider networks where the interarrival times are generally distributed but have exponential
or heavier tails. We show that these distributions can be decomposed into a mixture of sums of independent random variables such that at
least one of the components is exponentially distributed. This allows an easily implementable embedded regenerative structure in the Markov process. 
We show that under mild conditions on the network primitives, the regenerative mean and standard deviation estimators are consistent and satisfy a joint 
central limit theorem useful for  constructing asymptotically valid confidence intervals. We also show that amongst all such interarrival time decompositions, 
the one with the largest mean exponential component minimizes the asymptotic variance of the standard deviation estimator.
\end{abstract}
\vspace{-2mm}
\section{Introduction}

As is well known, a regenerative process is a stochastic process with a sequence of random time instants (known as regeneration times or an embedded renewal process)
such that at these instants the process probabilistically regenerates itself [\cite{WLS55}; \cite{Asm03}]. 
Regenerative simulation exploits this structure by generating independent and identically distributed ($i.i.d.$) cycles via simulation
and extracting consistent estimators of the steady-state performance measures from them. For a detailed review of the regenerative simulation, see  \cite{HN06} 
(also see, e.g, [\cite{GI87}; \cite{GP94}; \cite{HG01}; \cite{Asm03}]).\\

A typical multiclass queueing network consists of a set of single server stations and it is populated by customers of a finite number of classes. 
Every customer visits stations in a sequential order, receives service at each station visited, possibly with a queueing delay, and then  moves on. 
An \textit{open} queueing network is one in which customers arrive from the outside, each customer receives a finite number of 
services possibly at different stations, and then leaves the network
(see, e.g., [\cite{CH95}; \cite{DM95}; \cite{Dai95}]).
In this paper, we focus on multiclass open queueing networks and hereafter, refer to them as multiclass networks. 
These models find wide applications in communication, computer and manufacturing networks (see, e.g., [\cite{Kumar89}; \cite{Kumar90}; \cite{Kumar91}]). 
Typically, regenerative simulation is considered applicable to queueing networks when the interarrival times are exponentially distributed, 
where for example, instants of arrivals to an empty network denote a sequence of regeneration times. However, a-priori it is not clear whether implementable regeneration 
schemes can be identified for a queueing network when the interarrival times are generally distributed. 
In this paper, we construct an implementable regenerative simulation method to estimate the 
steady-state performance measures of multiclass networks when, roughly speaking, 
the interarrival times are generally distributed and have exponential or heavier tails. 
Networks where the interarrival times have exponential or heavier tails, e.g., 
Gamma or Hyper-exponential distributions, are common  in practice (e.g, see \cite{LJ97}; \cite{GMN97}).
A preliminary version of this paper will appear in \cite{SS13}. \\ 

Typically, the performance measures of interest in such  networks  involve expectations of
random variables in steady-state (e.g., steady-state expected queue length). 
These performance measures, aside from the simple settings, do not have an analytical solution, and one needs to resort to numerical methods to estimate them.
If the underlying Markov process visits a particular set of states infinitely often with probability one such that the process regenerates 
every time it leaves the set then that set can be used to conduct regenerative simulation to estimate the steady-state performance measures.
For example, in a stable $GI/G/1$ queue, system becomes empty infinitely often and it regenerates every time an arrival finds the system empty.
However,  in most queueing networks, such states are difficult to identify. \\

\cite{Dai95} establishes that under stability of the fluid limit model of a multiclass network, the associated Markov process is positive Harris recurrent.
\cite{AN78} and \cite{EN78} develop a splitting technique on the transition kernel of a positive Harris recurrent
Markov process that guarantees the existence of regenerations in the process (also refer to \cite{MT09} and \cite{GI81}).
However, identification of the regeneration instants involves explicit knowledge of the transition kernel that is typically difficult to
compute (see, e.g, \cite{HG99}).  A-priori it is not clear if implementable regeneration schemes can be developed for multiclass networks 
without the explicit knowledge of probability transition kernel.
Our analysis relies on the observation that random variables with exponential or heavier tails
can be re-expressed as a mixture of sums of independent random variables where at least one of the components is exponentially distributed.
This allows us to embed a regenerative structure in  queuing networks where the interarrival times are generally distributed with
exponential or heavier tails. In particular, the instants when arrivals of a particular class
of customers, say Class~1, that find the system empty and all the other classes are in \textit{exponential phase}, are regeneration times.
These regenerations typically exist even when Class~1 interarrival times have a super-exponential distribution such as a uniform distribution.
In addition to these regenerations, we also propose an alternative regenerative structure that exists when Class~1 interarrival times also have an 
exponential or heavier tail distribution. This corresponds to an arrival to an empty network just after all the classes are in exponential phase. \\

\cite{GI87} develop finite moment conditions on regenerative cycles under which the mean and the standard deviation estimators from   
regenerative processes can be shown to be consistent and satisfy a joint central limit theorem (CLT).
From simulation viewpoint, the latter is crucially useful as it allows construction of asymptotically valid error bounds or
confidence intervals along with a point estimator for the performance measure of interest.
Our another contribution is to show  that the existence of $p^{th}$ moments of the interarrival and service times, under mild stability conditions, 
is sufficient to guarantee the existence of $p^{th}$ moments of the regeneration intervals. \\

When there is more than one regenerative structure in a network, identifying the one with the best statistical properties becomes important.
The numerical examples in \cite{GI87} suggest that, in general, selection of regenerative structure with minimum mean return time need not minimize
the \textit{asymptotic variance of the standard deviation estimator (AVSDE)}. 
As is well known, classical regenerative processes are stochastic processes that can be viewed as concatenation of $i.i.d.$ cycles.
\cite{JC94} shows that combining adjacent cycles of classically regenerative process
increases the AVSDE. \cite{ACG95} generalize this result by showing that 
if one regenerative structure of a classically regenerative process is a subsequence of another, then the AVSDE associated with the original structure 
is larger than that associated with the subsequence. We generalize this result and show that, in our framework, 
the selection of interarrival time decomposition with the largest mean exponential component results in minimum AVSDE.\\

A potential drawback of our approach is that regenerations in a large network maybe infrequent since an arrival of a particular 
class not only needs to find the network empty but also that all the other classes in the exponential phase, 
thus restricting its application to smaller networks. However, with the advent of parallel computing it may be feasible to implement proposed ideas
to somewhat larger networks (see,~e.g,~\cite{PGHP91}).\\

The remaining paper is organized as follows: In Section \ref{Reg_process}, we review both the classical and general regenerative simulation, and the associated CLTs.
In Section \ref{sec_dec}, we show that under mild conditions, a random variable with an exponential or heavier tail distribution 
can be re-expressed as a mixture of sums of independent random variables where one of the constituent random variables has an exponential distribution.
Formal construction of multiclass network is presented in Section \ref{MCOQN}.
Section \ref{MomRevTimes} presents a regenerative simulation methodology for  these networks and establishes the finiteness of moments of regeneration intervals under 
appropriate assumptions on network primitives. In Section \ref{Freq_reg}, we establish that if one regenerative structure is a subsequence of another then the 
AVSDE associated with the former regenerative structure is at least as large as than that associated with the latter, and present two important applications 
of this result. Section \ref{Reg_sim} illustrates the proposed simulation method using simple numerical examples. 
Unless otherwise stated, all the proofs are presented in Section~\ref{sec:proofs}. 
\section{Notation and Terminology}
We first introduce notation that will be used throughout the paper. Assume that all the random variables and stochastic processes are defined on a common 
probability space $(\Omega, \mathcal{F}, P)$. For any metric space $\mathcal{S}$,  $\mathcal{B_S}$ denotes the Borel $\sigma$-algebra on it. 
The notation $X \sim F$ is used to denote that the distribution of a random variable $X$ is~$F$. 
Exponential  distribution with rate $\lambda > 0$ is denoted by $Exp(\lambda)$. We write $X_1 \stackrel{d}{=} X_2$ to denote the equivalence of distributions of 
random variables $X_1$ and $X_2$. The function $\phi_X(t) := \int_{-\infty}^{\infty}e^{ixt}dF(x)$ denotes the characteristic function of 
a random variable $X$ when $X \sim F$,
and we say that $\phi_X$ is absolutely integrable if $\int_{-\infty}^{\infty} |\phi_X(t)|dt < \infty$. A subprobability measure $\nu$ is a \textit{component} 
of the distribution of a $\mathcal{S}$-valued random variable $X$ if $P\left( X \in A \right) \geq \nu(A)$ for every Borel set $A \in \mathcal{B_S}$ 
(refer to \cite{THR2K}). The indicator function is denoted by $\mathbb{I}(\cdot)$, which is~1 if the argument is true and 0 otherwise.
We say that a probability distribution is \textit{lattice} if it is concentrated on a set of points of the form $a + nh$, where
  $h > 0$, $a$ is a real value and $n = 0, \pm 1, \pm 2, \dots$. Any non-zero $\sigma$-finite measure $\pi$ is an \textit{invariant} measure of the process
$X= \left\{ X(t): t \geq 0 \right\}$ if  $\pi(B) = \int_{\mathcal{X}} P\left(X(t) \in B | X(0) = y\right)\pi(dy)$, for all 
$B \in \mathcal{B_X}, t\geq 0$, where $\mathcal{X}$ is the state space of $X$. If every invariant measure is a positive scalar multiple of $\pi$, 
then it is well known that $\pi$ is the \textit{unique} invariant probability measure of the process $X$.  
We assume that all the processes considered in this paper are c\`{a}dl\`{a}g (right continuous paths with left limits).

\section{Regenerative Simulation}
\label{Reg_process}
As is well known,  a sequence of random variables $0 = T_{-1}\leq T_0 < T_1 < T_2 < \cdots$ is called a \textit{renewal process}
if the sequence of intervals $\left\{T_n - T_{n-1}; n \geq 1\right\}$ is an i.i.d. sequence and independent of $T_0$.
It is said to be  \textit{non-delayed} renewal process if  $T_0 = 0$; otherwise, it is a \textit{delayed} renewal process with delay $T_0$.\\

The following definition of regenerative process is based on \cite{Asm03} (also refer to \cite{THR83} where 
it is known as wide-sense regenerative process).
\begin{definition}
A stochastic process $Y = \left\{ Y(t): t \geq 0 \right\}$ is called regenerative if there exists a renewal process $0 \leq T_0 < T_1 < \cdots$ such that 
(i) $\left\{ Y(T_n + s):  s \geq 0 \right\}$ is independent of $\left\{  T_0,\dots,T_n\right\}$ and
	(ii) $\left\{ Y(T_n + s):  s \geq 0 \right\}$ is stochastically equivalent to $\left\{ Y(T_0 + s):  s \geq 0 \right\}$ for $n \geq 0$.
\end{definition}
The sequence $T_0, T_1, \dots$ is referred to as a sequence of \textit{regeneration times}.
In addition, if the regeneration cycles, $ \left\{ \left\{ Y(s): T_{n-1} \leq s < T_n \right\}, n\geq 0\right\}$, are independent
then the process is known as \textit{classically} regenerative.
We refer to $Y$ as a delayed (respectively, non-delayed) regenerative process if the associated renewal process
is delayed (respectively, non-delayed). 
Conceptually, one may think of a non-delayed classically regenerative process as a concatenation of $i.i.d.$ cycles. 
As is well known, when $T_0$ is a proper random variable (that is, $P(T_0 < \infty) \equiv 1$), the steady-state behavior of the process does not 
depend on the first cycle. To simplify the discussion, throughout this section, we assume that all the regenerative processes are non-delayed. \\  

The following theorem is important to our analysis. Refer, e.g., to  Theorem 1.2 in Chapter VI of \cite{Asm03}) for proof. 
Let $\mathcal{Y}$ be the state space of the process~$Y$, and for any distribution $\nu$ on $\left( \mathcal{Y}, \mathcal{B_Y} \right)$, let $ \mathbb{E}_{\nu} (\cdot) := \int_{\mathcal{Y}}\mathbb{E}_y(\cdot) \nu(dy)$, where
$E_y$ is the expectation operator associated with the probability measure $P_y$ that represents the law of the process when it starts in state $y$, that is, $Y(0) = y$.
\begin{theorem}
 \label{reg_2_thm}
  Suppose that $Y$ is a regenerative process with regeneration times $0 =  T_0 < T_1 < \cdots$ and the distribution of the first cycle
  length $T_1$ is non-lattice with finite mean. Then, the steady-state distribution $\pi$ exists and for any non-negative real valued function $h$,
\begin{align*}
   \mathbb{E}_{\pi}\left[ h\left(Y(t)\right)\right] &= \frac{1}{\mathbb{E}_{\varphi}\left[T_1\right]}\mathbb{E}_{\varphi}\left[ \int_0^{T_1}h\left(Y(s)\right)ds\right],
\end{align*}
where $\varphi$ is the distribution of the initial state $Y(0)$.
\end{theorem}
\begin{example}
\label{Ex_gen_reg}
\normalfont
Consider a simple queueing system consisting of a single server that serves two streams of arrivals. The interarrival and service times of each stream are independent
$i.i.d.$ sequences. Let
\begin{align*}
 Y(t) = \left(Q(t), U(t), V(t) \right), \ t \geq 0,
\end{align*}
where $Q(t) = \left[Q_1(t), Q_2(t)\right]$ is the queue length process with $Q_k(t)$ being the number of Stream~$k$ customers in the system, 
$U(t) = \left[U_1(t), U_2(t)\right]$ with $U_k(t)$ being the remaining time until the next Stream~$k$ arrival and 
$V(t) = \left[V_1(t), V_2(t)\right]$ with $V_k(t)$ being the remaining service time of Stream~$k$ that is under service at time $t$.\\

Suppose that the Stream~2 has exponentially distributed interarrival times. Then every instant when an arrival of Stream~1 finds the system empty is a regeneration
time of $Y$. One can easily verify that these regenerations are not classical regenerations.
\qed
\end{example}

Associated with every regenerative process there exists a classical regenerative process with the same steady-state performance measures. To see this, suppose that  $Y$
is a regenerative process with regeneration times $0 =  T_0 < T_1 < \cdots$. Now construct a classically regenerative process $X$ with regeneration times 
$0 =  S_0 < S_1 < \cdots$ by concatenating 
the sequence of $i.i.d.$ cycles $\left\{X(t): S_{n-1} \leq t < S_n\right\}, \ n \geq 1$ such that each cycle is probabilistically equivalent to the cycle 
$\left\{Y(t): 0 \leq t < T_1\right\}$.
Then, under the hypothesis of the above theorem,
\begin{align*}
   \mathbb{E}_{\pi}\left[ h\left(Y(t)\right)\right] = \frac{1}{\mathbb{E}_{\varphi}\left[T_1\right]}\mathbb{E}_{\varphi}\left[ \int_0^{T_1}h\left(Y(s)\right)ds\right]
                                                    = \frac{1}{\mathbb{E}_{\varphi}\left[S_1\right]}\mathbb{E}_{\varphi}\left[ \int_0^{S_1}h\left(X(s)\right)ds\right].
\end{align*}
Thus, one can exploit the classical regenerative property of $X$ to estimate the steady-state performance measures associated with $Y$. A classically regenerative process 
that associated with the process $Y$ in Example~\ref{Ex_gen_reg} can be constructed by generating a new arrival clock of Stream~2 independently of the past at every instant a Stream~1 customer 
arrival finds the system empty. This is possible because the interarrival times of Stream~2 are exponentially distributed and at these instants,  
arrival and service clocks of Stream~1 are generated independently of all previous history.\\

Often a steady-state performance measure of interest has the form $\bar{r} := \int_{\mathcal{Y}}h(y)\pi(dy)$,
where $h$ is a non-negative real valued function defined on $\mathcal{Y}$.
A natural simulation estimator for $\bar{r}$ is $r(t) := \frac{1}{t} \int_0^t h\left( X(s) \right)ds$.
Under minimal conditions, it can be shown that  $r(t) \longrightarrow \bar{r}, \ \ a.s.\ \text{as} \ \ t \rightarrow \infty.$\\

Now define $\beta(t) := \frac{\sum_{i=1}^{N(t)} R_i}{\sum_{i=1}^{N(t)} \tau_i}$ and 
$s(t) := \sqrt{\frac{\sum_{i=1}^{N(t)}\left( R_i - \beta(t)\tau_i \right)^2}{\sum_{i=1}^{N(t)} \tau_i}}$, for $t \geq 0$,
where the counting process $N(t)~:=~\max\left\{ n \geq 0: S_n \leq t \right\}$ which counts the number of regenerations that have occurred till time $t$, and
for each $i \geq 1$, $ R_i := \int_{S_{i-1}}^{S_i} h\left( X(s) \right)ds$ and $\tau_i~:=~S_i~-~S_{i-1}$. 
Set $W_i = R_i - \bar{r}\tau_i$ for each $i \geq 1$. For the proof of Theorem~\ref{reg_1_thm}, see, e.g,  \cite{GI87}.\\
\begin{theorem}
	\label{reg_1_thm}
	Let $X$ be a classically regenerative process with regeneration times $0= S_0 < S_1 < \cdots$ and invariant probability measure $\pi$.
	Set $\varphi(dy) := P\left( X(0) \in dy \right)$ and assume that $h \geq 0$.
	\begin{enumerate}
		\item[(i)] If $\mathbb{E}_\varphi \tau_1 < \infty $, then $\bar{r} = \frac{\mathbb{E}_{\varphi} R_1}{\mathbb{E}_{\varphi} \tau_1}$ and as 
		              $t \rightarrow \infty$
			\begin{align}
				r(t) &\longrightarrow \bar{r}, \ \ a.s.,\label{SLLN_r}\\
				\beta(t) &\longrightarrow \bar{r}, \ \ a.s. \label{SLLN_beta}
			\end{align}
		\item[(ii)] If $\mathbb{E}_\varphi \left[R_1^2 + \tau_1^2 \right] < \infty $, then as $t \rightarrow \infty,$			
			\begin{align}
			        s(t) \longrightarrow \sigma, \ &a.s.,\label{SLLN_s}\\
			        t^{\frac{1}{2}} \left( r(t)  -\bar{r}\right) &\Rightarrow \sigma\mathcal{N}(0,1),\label{CLT_r2}\\
			        t^{\frac{1}{2}} \frac{\left( r(t)  -\bar{r}\right)}{s(t)} &\Rightarrow \mathcal{N}(0,1),\label{CLT_r3}\\
		                t^{\frac{1}{2}} \frac{\left( \beta(t)  -\bar{r}\right)}{s(t)} &\Rightarrow \mathcal{N}(0,1),\label{CLT_beta}
                        \end{align}
			where $\sigma^2 = \frac{\mathbb{E}_{\varphi} W_1^2}{\mathbb{E}_{\varphi} \tau_1}$
			 is known as the \textit{time-average variance constant (TAVC)} of $h(X(\cdot))$.
		\item[(iii)] If $\mathbb{E}_\varphi \left[R_1^4 + \tau_1^4 \right] < \infty $ and $\sigma > 0$, then as $t \rightarrow \infty,$
			\begin{align}
			        t^{\frac{1}{2}} \left( r(t)  - \bar{r}, s(t) - \sigma \right) &\Rightarrow \mathcal{N}(\stackrel{\rightarrow}{0},\mathcal{K}),\label{JCLT_rs}\\
		                t^{\frac{1}{2}} \left( \beta(t)  - \bar{r}, s(t) - \sigma \right) &\Rightarrow \mathcal{N}(\stackrel{\rightarrow}{0},\mathcal{K}),\label{JCLT_bs}
                        \end{align}
			where
			\begin{align*}
				\mathcal{K} = \frac{1}{\mathbb{E}_\varphi \tau_1} \left[ \begin{array}{cc}
				\mathbb{E}_\varphi \left(W_1^2\right) & \frac{\mathbb{E}_\varphi\left[\left( A_1 - b W_1\right)W_1\right]}{2\sigma} \\
				\frac{\mathbb{E}_\varphi\left[\left( A_1 - b W_1\right)W_1\right]}{2\sigma} &
      				\frac{\mathbb{E}_\varphi \left[\left(A_1 - b W_1\right)^2\right]}{4\sigma^2}
					   \end{array}\right],
			\end{align*}
			$A_i = W_i^2 - \sigma^2 \tau_i$, $b = 2\mathbb{E}_\varphi\left( W_1 \tau_1\right)/\mathbb{E}_\varphi \tau_1$ and \
			$\mathcal{N}(\stackrel{\rightarrow}{0},\mathcal{K})$ represents a multivariate normal random variable with the covariance matrix $\mathcal{K}$.
	\end{enumerate}
\end{theorem}
\ \\
In the above theorem, $(i)$ constitutes the strong law of large numbers  ($SLLN$) for point estimators $r(t)$ and $\beta(t)$.
As is well known, $\beta(t)$ is the natural regenerative simulation estimator of $\bar{r}$. Part $(ii)$ constitutes $SLLN$ for the standard deviation estimator
$s(t)$ and the central limit theorems (CLTs) for $r(t)$ and $\beta(t)$. From (\ref{SLLN_s}) and (\ref{JCLT_bs}),
it is clear that $s(t)$ is an estimator of the asymptotic standard deviation $\sigma$ and it converges  at rate $\sqrt{t}$.
Notice from  (\ref{CLT_r2}) that, associated with the estimator $r(t)$, 
the asymptotic $100(1 - \delta)\%$ confidence interval for $\bar{r}$ is $\left[ r(t) - \frac{z\sigma}{\sqrt{t}}, r(t) + \frac{z\sigma}{\sqrt{t}}\right] $,
where $z$ solves the equation  $P\left(-z \leq \mathcal{N}(0,1) \leq z \right) = 1- \delta$.
On the other hand, from (\ref{SLLN_beta}), $\beta(t)$ is another point estimator for $\bar{r}$ and it \textit{depends} on the regenerative structure.
In this case, the associated $100(1 - \delta)\%$ confidence interval is 
$\left[ \beta(t) - \frac{z\sigma}{\sqrt{t}},\ \beta(t) + \frac{z\sigma}{\sqrt{t}}\right] $. In practice, $\sigma$ is unknown and must be estimated. 
A natural regenerative standard deviation estimator is $s(t)$ and thus, replacing $\sigma$ by $s(t)$, the associated asymptotic $100(1~-~\delta)\%$ confidence interval is 
$\left[ \beta(t) - \frac{z s(t)}{\sqrt{t}},\ \beta(t) + \frac{z s(t)}{\sqrt{t}}\right]$.
Finally, part $(iii)$ constitutes the joint $CLT$ for both $\left(r(t),s(t)\right)$ and $\left(\beta(t),s(t)\right)$ and describes the asymptotic variance of the 
standard deviation estimator $s(t)$. 
Now it is clear that, to construct valid confidence intervals, one sufficient condition is to ensure that 
$\mathbb{E}_\varphi \left[R_1^4 + \tau_1^4 \right] < \infty $. To establish these moments, in Sections \ref{moment_tau} and \ref{section:moment_Y}, 
we study sufficient moment conditions on interarrival and service times in the context of multiclass networks. 

\section{Extracting Exponential Component in a Distribution}
\label{sec_dec}

In this section, we show that under mild conditions any random variable with an exponential or heavier tail distribution can be re-expressed as a mixture of sums of
independent random variables such that one of the constituent random variables has an exponential distribution. 
We then observe this decomposition for well-known distributions such as Pareto, Weibull, Gamma, etc.\\

Let $f$ be the probability density function of a real valued random variable. We say that $f \in \mathcal{H}$, if there exists an $a \in [-\infty, \infty)$
such that $f(x) = 0$,  for all $x < a$ (if $a > -\infty $), $f$ is differentiable on $(a,\infty)$ and
\begin{equation}
 \label{lambda_star}
 \lambda_f := \sup_{y \in (a,\infty)}\left(-\frac{f'(y)}{f(y)}\right) \in (0, \infty).
\end{equation}
Examples of densities in $\mathcal{H}$ are discussed later. For a density function $f \in \mathcal{H}$, define 
\begin{align}
   \label{G_lam_eq}
	G^f_\lambda(x) := F(x) + \frac{f(x)}{\lambda}, \ x \in \reals
\end{align}
where the distribution function $F(x) = \int_{-\infty}^x f(y)dy$, and $\lambda > 0$.
Theorem~\ref{thm_dec1} is one of our key results 
and the proof primarily depends on the fact that any characteristic function uniquely identifies the associated probability distribution.
\begin{theorem}
\label{thm_dec1}
Suppose that $\xi$ is a random variable with density $f \in \mathcal{H}$ and $E \sim Exp(\lambda)$ for $\lambda \in \left(0,\infty\right)$. 
If $\lambda \geq \lambda_f$ then there exists a random variable $Z$ independent of $E$ such that
 \begin{equation}
  \label{EZ_dec}
 \xi \stackrel{d}{=} E + Z.
\end{equation}
Conversely, if $\phi_\xi$ is absolutely integrable, $f'$ is continuous on $(a,\infty)$ and (\ref{EZ_dec}) holds for some $Z$ independent of $E$ 
then $\lambda \geq \lambda_f$. In this case $Z \sim G^f_\lambda$.
\end{theorem}
\begin{remark}
\normalfont
 \label{remark:superExp}
 Theorem~\ref{thm_dec1} ensures that if (\ref{lambda_star}) is true then $\xi$ must have an exponential or heavier tail distribution.
\end{remark}
Now we consider some  practically important classes of  distributions with exponential or heavier tails.

\begin{example}[\textbf{Lognormal}]
\normalfont
 Suppose $\xi = \exp\left[ \mathcal{N}\left(\mu, \sigma^2\right)\right]$, that is, $\xi$ has a Lognormal distribution, where $\mathcal{N}\left(\mu, \sigma^2\right)$
 represents a normal random variable with mean $\mu$ and variance $\sigma^2~>~0$. Then, its pdf
 \begin{align*}
 g(x) = \frac{1}{x\sqrt{2\pi \sigma^2}} \exp\left[{-\frac{(\log x - \mu)^2}{2\sigma^2}}\right], \ x > 0.
\end{align*}
Therefore $ - \frac{g'(x)}{g(x)} = \frac{1}{x}\left[ 1 + \frac{\log x - \mu}{\sigma^2}\right]$. It is easy to compute that 
$\displaystyle\lambda_g =  \frac{\exp\left(\sigma^2 -(\mu +1)\right)}{\sigma^2}$
and it is achieved at $x = \exp\left(\mu + 1 - \sigma^2 \right)$. Thus, the mean of the maximum exponential component 
that can be extracted in this manner equals $\sigma^2 \exp((\mu+1)-\sigma^2)$.  
As is well known, $\mathbb{E}\left[\xi\right] = \exp(\mu+\sigma^2/2)$, so that the ratio of the former to the latter equals
$\sigma^2 \exp(1-3/2 \sigma^2)$. Interestingly, this is independent of $\mu$ and is maximized at $2/3$ 
when $\sigma^2=2/3$. 
\qed
\end{example}

\begin{example}[\textbf{Gamma}]
\normalfont
\label{Gamma_dec}
 Suppose $\xi$ is $Gamma(\alpha,\gamma)$ distributed with pdf $g(x)~=~\frac{\gamma^{\alpha} x^{\alpha-1} \exp(-\gamma x)}{\Gamma(\alpha)}$ for
 $x \geq 0$ and $g(x)=0$, otherwise, where  $\gamma > 0$ and $\alpha \geq 1$. As is well known, Gamma distribution has an exponential tail.
 It is easily seen  that  $ \lambda_g= \sup \left(-\frac{g'(y)}{g(y)}\right) = \gamma$ and hence, we achieve the desired decomposition.
 \qed
\end{example}

\begin{example}[\textbf{Log-Convex densities, Pareto}]
\normalfont
\label{log_convex_ex}
 Let $\xi$ be a random variable with log-convex density $g$ that has support $[a,\infty)$ and is differentiable on $(a, \infty)$. 
 Since $-\log g(x)$ is concave function of $x$,
  $\frac{-g'(x)}{g(x)} = \frac{d}{dx}\left( -\log g(x)\right)$
 is a monotonically non-increasing function of $x$ (because $\frac{d^2}{dx^2}\left( -\log g(x)\right) \leq 0$). Furthermore,
  $\lambda_g = \frac{-g'(a)}{g(a)}$ and when $\lambda_g \in (0, \infty)$,  we achieve the desired decomposition. \\
  
  For instance, suppose that $\xi$ has Pareto distribution with the shape parameter $\alpha > 0$ and the scale parameter $\gamma > 0$,
that is, $P\left( \xi > x \right) = \frac{1}{(1 + \gamma x)^{\alpha}},\ \text{ for  } x \geq 0.$
Then its pdf  $g(x) = \frac{\alpha \gamma}{(1 + \gamma x)^{\alpha+1}}, \ x \geq 0$.
 Now it is easy to see that $\mathbb{E}\left[\xi \right] = [(\alpha-1) \gamma]^{-1}$. Furthermore,
 $\lambda_g= [(\alpha+1) \gamma ]$ so that the maximum mean of the extracted exponential component
 equals $[(\alpha+1) \gamma ]^{-1}$. In particular, the ratio of this extracted mean to the
 total mean equals $\frac{\alpha-1}{\alpha+1}$. It is independent of $\gamma$ and increases from zero to 1
 as $\alpha$ increases from 1 to~$\infty$. 
 \qed
 \end{example}

\begin{example}[\textbf{Hyper-exponential}]
\normalfont
Suppose that $\xi$ has the pdf 
\begin{align*}
 g(x) = p_1 \lambda_1 e^{-\lambda_1 x} + p_2 \lambda_2 e^{-\lambda_2 x}, \ x \geq 0,
\end{align*}
 where $p_1$, $p_2 > 0$,  $p_1 + p_2 = 1$, and $\lambda_1 > \lambda_2 > 0$.
In this case, any one of the two exponential components can be used for regenerative simulation (described later). 
It is reasonable to consider a component with the largest contribution to mean, that is, the one that maximizes
$\{\frac{p_i}{\lambda_i}\} $  for $i=1,2$.   Note that under
certain conditions it is possible to extract exponential distribution with mean greater than $\max\left\{\frac{p_1}{\lambda_1}, \frac{p_2}{\lambda_2}\right\}$.
Now observe that
\begin{align*}
-\frac{g'(x)}{g(x)} = \frac{p_1 \lambda_1^2 e^{-\lambda_1 x} + p_2 \lambda_2^2 e^{-\lambda_2 x}}{p_1 \lambda_1 e^{-\lambda_1 x} + p_2 \lambda_2 e^{-\lambda_2 x}}
\end{align*}
is a monotonically decreasing function that achieves maximum
at $x = 0$ and  reaches $\lambda_2$ as $x \nearrow \infty$. Therefore
$\lambda_g = \frac{p_1 \lambda_1^2 + p_2 \lambda_2^2}{p_1 \lambda_1 + p_2 \lambda_2}$ and since $p_1 = 1 - p_2$, it is easy to verify that
$ e(p_1) := \frac{1}{\lambda_g} = \frac{p_1 \lambda_1 + p_2 \lambda_2}{p_1 \lambda_1^2 + p_2 \lambda_2^2}$ is a monotonically
decreasing function of $p_1 \in [0,1]$ and $e(p_1)~\geq~\frac{p_1}{\lambda_1}$ for any value of $p_1$.
Using simple analysis, we can find that $e(p_1) \geq \max\left\{\frac{p_1}{\lambda_1}, \frac{p_2}{\lambda_2}\right\}$ holds, if either
(i) $\lambda_1\lambda_2 > \lambda_1^2 - \lambda_2^2$ or (ii) $\lambda_1\lambda_2 \leq \lambda_1^2 - \lambda_2^2$ and $ p_1 \in \left[ 1  - \frac{\lambda_1\lambda_2}{\lambda_1^2 - \lambda_2^2}, 1\right]$.
\qed
\end{example}

The example below illustrates the fact that there exist some distributions which may not have densities that belong to the family $\mathcal{H}$,
but their components can be in $\mathcal{H}$ (with a scaling factor).

\begin{example}[\textbf{Weibull}]
\normalfont
\label{Weibull_dec1}
 Suppose $\xi$ has a $Weibull(\alpha, \gamma)$ distribution with the shape parameter $\alpha < 1$ and scale parameter $\gamma > 0$. Then its tail distribution
  \begin{align*}
  P\left( \xi > x \right) = \exp\left(-(\gamma x)^{\alpha}\right),\ x > 0.
 \end{align*}
 If $g$ is the associated pdf, then $\frac{-g'(x)}{g(x)} = \alpha\gamma^{\alpha} \frac{1}{x^{1- \alpha}}  + \frac{(1- \alpha)}{x}$.
  Since $\alpha < 1$, $\frac{-g'(x)}{g(x)}$ is a decreasing function of $x$ that reaches $\infty$ as $x \searrow 0$ and reaches zero as $x \nearrow \infty$.
 Thus $\lambda_g = \sup_{y \in (0,\infty)} \left(-\frac{g'(y)}{g(y)}\right) = \infty$ and is achieved at $y = 0$. Therefore, $\xi$ may not be decomposed in the form (\ref{EZ_dec}).
 However, $ \sup_{y \in (\hat{a},\infty)} \left(-\frac{g'(y)}{g(y)}\right) < \infty$, for any $\hat{a} > 0$.
 Hence by fixing $\hat{a} > 0$, and letting $f(x) = c g(x)$ for $x \geq \hat{a}$ and $f(x)= 0$ for $x < \hat{a}$, 
 we have $f \in \mathcal{H}$ and $g \geq f/c$, where $c = \frac{1}{P\left( \xi > \hat{a} \right)}$. 
 \qed
\end{example}

 Example~\ref{Weibull_dec1} motivates a more general framework for extracting an exponential component.
To see this, suppose that $G$ is the distribution of a real valued random variable $\xi$ such that for some $q \in (0,1]$ and $f \in \mathcal{H}$ and 
$qf$ is a component of $G$, that is, $G(dy) \geq qf(y)dy$.
Let $\hat{\xi} \sim f$ and fix $\lambda \geq \lambda_f$, then from Theorem~\ref{thm_dec1},
$\hat{\xi} \stackrel{d}{=} E + Z,$ where $E \sim Exp(\lambda)$, $Z \sim G^f_\lambda$ and they are independent. If $q = 1$ then $f$ is the density of $G$, hence
$\xi \stackrel{d}{=} \hat{\xi} \stackrel{d}{=} E + Z.$ But if $q < 1$, then we can let 
$H(x) = \frac{G(x) - q\int_{-\infty}^x f(y)dy}{1 - q},\ \ x \in \mathbb{R}$.
Clearly, $H$ is a probability distribution function and $G(x) = (1-q)H(x) + q \int_{-\infty}^x f(y)dy,\ \ x \in \mathbb{R}$.
In other words
\begin{equation}
    \label{Gen_dec}
    \xi \stackrel{d}{=} (1 - \beta) \tilde{\xi} + \beta(E + Z),
\end{equation}
where $\beta$ is a Bernoulli random variable with $P(\beta = 1) = q$, $\tilde{\xi} \sim H$ and $\tilde{\xi}$, $E$, $\beta$ and $Z$ are independent of each other.

\begin{example}[\textbf{Log-Convex with $\mathbf{\lambda_g = \infty}$}]
\normalfont
 Recall Example \ref{log_convex_ex}, where the density of the random variable $\xi$ is a log-convex function $g$ that has support
 $[a,\infty)$ and is differentiable on $(a, \infty)$.
 Suppose that $\lambda_g = \infty$ (for example, Weibull distribution).  Since $\frac{-g'(x)}{g(x)}$ is a decreasing function of $x$,
 we can choose $\hat{a}$ such that $\frac{-g'(\hat{a})}{g(\hat{a})} < \infty$. Let $f(x)~=~\frac{1}{q} g(x)$ for $x \geq \hat{a}$ and $f(x) = 0$ for $x < \hat{a}$,
 where $q = P\left(\xi \geq \hat{a} \right)$. Clearly $f \in \mathcal{H}$ and $qf$ is a component of $g$. Hence $\xi$ can have decomposition of the form (\ref{Gen_dec}).\\
  
 Recall Example~\ref{Weibull_dec1}, that is, $\xi \sim Weibull(\alpha, \gamma)$ with $\alpha < 1$ and $\gamma > 0$. It is reasonable to choose $\hat{a} \geq 0$ 
 that maximizes the exponential mean contribution, $\frac{q}{\lambda_f}$. It turns out that this maximization is achieved at 
 $\hat{a} = \frac{1}{\gamma \alpha^\alpha}\left(1 - \alpha \right)^{\frac{1}{2\alpha}}$ and, in particular, it can be shown that the ratio of this extracted 
 exponential mean to the total mean is independent of~$\gamma$ (the expression is mathematically complex and is omitted).
 \qed
 \end{example}

 \section{Multiclass Open Queueing Networks}
\label{MCOQN}
The following notation is needed to characterize a  multiclass open queueing network:
\begin{table}[h]
\centering
\begin{tabular}{p{3cm} p{10.2cm}}
$d$:                           & number of single server stations indexed with $i=1,2,....,d$.\\
\\
$K$:                           & number of customer classes in the network.\\
\\
$\{\xi_{k,n}:n \geq 1\}$:      & sequence of exogenous interarrival times of Class~$k$. If $\xi_{k,1} = \infty$, then we say that
                                 external arrival process of Class~$k$ is \textit{null} or simply, Class~$k$ is null exogenous.\\
\\
$s(k)$:                        & station at which Class~$k$ customers take service.\\
\\
$\{\eta_{k,n}:n \geq 1\}$:     &  sequence of service times of Class~$k$ at station $s(k)$.\\
\\
$P_{kl}$:                      & probability of Class~$k$ customer becoming Class~$l$ customer upon completion of service at station $s(k)$,
                                independent of all previous history; the customer exits the network with probability $1~-~\sum_l P_{kl}$. \\
\\
$L$:                           & number of classes with \textit{non-null} exogenous arrivals. \\
\end{tabular}
\end{table}

We assume that $(I - P')^{-1} = (I + P + P^2 + \dots )'$ exists, where $P'$ is the transpose of the matrix~$P$.
Since $(k,j)$-element of $(I - P')^{-1}$ is the expected number of times a Class~$k$ customer visits Class~$l$ during its stay in the network,
every customer who enter into the network will leave it eventually. Hence, the network described above is an open queueing network.\\

Like the $GI/G/1$ queue, any queueing network with a single non-null exogenous class regenerates at every
instant when a customer arrival finds the system empty. Hence a regenerative structure trivially exists in these networks.
So, without loss of generality we assume that $L \geq 2$ and first $L$ classes are non-null exogenous.

\subsection{Assumptions}
\label{assumptions}
 Throughout the paper, we make the following assumptions on the network primitives.
\begin{enumerate}
 \item [(A1)] $\boldsymbol{\xi_1,\xi_2, \dots , \xi_L, \eta_1, \eta_2,  \dots , \eta_K}$ are $i.i.d.$
             sequences and mutually independent, \\ where $\boldsymbol{\xi_i} = \left\{ \xi_{i,n} : n \geq 1 \right\}$ and
             $\boldsymbol{\eta_i} = \left\{ \eta_{i,n} : n \geq 1 \right\}$.
 \item[(A2)] There exists $p \geq 1$ such that \\ $0 < \mathbb{E}[(\xi_{k,1})^p] < \infty$ for $k = 1,\dots,L$ and 
	     $0< \mathbb{E}[(\eta_{k,1})^p]< \infty$ for $k = 1,\dots, K$.
 \item[(A3)] For each $k = 2, \cdots, L$, there exists $f_k \in \mathcal{H}$ such that $P\left(\xi_{k,1} \in dx \right) \geq \bar{q}_k f_k(x)dx$ for some $\bar{q}_k > 0$,
             where $\mathcal{H}$ is defined in Section \ref{sec_dec}.
 \item[(A4)] Distribution of $\xi_{1,1}$ is spreadout, that is, $P(\xi_{1,1} + \cdots +\xi_{1,j} \in dx) \geq \bar{q}(x)dx$ 
             for some non-negative function $\bar{q}(x)$ and integer $j$ such that $\int_0^\infty \bar{q}(x)dx > 0$. 
 \item[(A5)] $P(\xi_{1,1} \geq x \text{ and } \sum_{l=1}^K \eta_{l,1} < x) > 0$ for some $x$.
\end{enumerate}
Assumption (A1) is standard. It is clear from Section~\ref{sec_dec} that (A3) holds only when the interarrival times have either exponential or heavier tails.
(A3) and (A4) are useful in establishing the ergodicity of the network as well as identifying regenerative structures in it, 
while (A2) and (A5) are useful in establishing finite moments of regeneration intervals  as we see in the later sections. Notice that (A5) trivially holds when the
interarrival times of Class~1 are unbounded or the service time distribution of every class has support in every neighborhood of 0. \\

We  denote the common distribution of the interarrival times of Class~$k$ by $F_k$ with the average arrival rate
$\alpha_k := \frac{1}{\mathbb{E}[\xi_{k,1}]}$, where $k\in\{1,2,\dots L \}$.
Similarly,  the common distribution of the service times of Class~$k$ is denoted by $H_k$ with the average service rate
$\mu_k := \frac{1}{\mathbb{E}[\eta_{k,1}]}$, where $k\in\{1,2,\dots K \}$.
Let $\mathcal{C}_i := \{k: s(k) = i\}$ be the \textit{constituency} for station $i \in \{1,2, \cdots, d\}$. By letting $\sigma := (I - P')^{-1} \alpha$, 
one can interpret $\sigma_k$ as the \textit{effective} arrival rate to Class~$k$. Then $\rho_i := \sum_{k \in \mathcal{C}_i} \frac{\sigma_k}{\mu_k}$ is 
the \textit{nominal load} for server $i \in \{1,2,\dots,d \}$ per unit time. 
From (A3) and (\ref{Gen_dec}), without loss of generality, we can write 
\begin{equation}
	\label{TwoMP2}
        \xi_{k,n} = (1 - \beta_{k,n})\tilde{\xi}_{k,n} + \beta_{k,n} (E_{k,n} + Z_{k,n} ), \ k = 2, \dots , L, \ n \in \integers_+,
\end{equation}
where $E_{k,n} \sim Exp(\lambda_k)$
for some $\lambda_k \geq \lambda_{f_k}$, $\beta_{k,n}$ is a Bernoulli random variable
with $P(\beta_{k,n} = 1) = \bar{q}_k$, $Z_{k,n} \sim G^{f_k}_{\lambda_k}$ and 
$\tilde{\xi}_{k,n} \sim \displaystyle \frac{F_k(x) - \bar{q}_k\int_0^x f_k(y)dy}{1- \bar{q}_k}, \ x \geq 0$; and $\{ \tilde{\xi}_{k,n} : n\geq~0\}$, 
$\left\{ E_{k,n} : n\geq 0\right\}$, $\left\{ \beta_{k,n} : n\geq 0\right\}$ and $\left\{Z_{k,n} : n\geq 0\right\}$ are $i.i.d.$ sequences and independent of each other.

\subsection{Markov Process}
\label{Markov_process}
Now, to describe the network, we propose the following Markov process that splits the interarrival times into two components:
\begin{eqnarray}
\label{Proc_Y}
	Y(t) &=& (Q(t),U(t), V(t)),
\end{eqnarray}
where $Q(t) = [Q_1(t), Q_2(t), \cdots , Q_K(t)]'\in \integers_+^K$. The process $Q_k(t)$ captures the number of Class~$k$ customers in the network at time $t$
or, more generally, it can capture positions of every Class~$k$ customer present at station $s(k)$ (in the later case $Q_k(t)$ is an infinite dimensional vector). 
Hereafter, the notations $\|Q_k(t)\|$ and $\|Q(t)\|$ denotes, respectively, the number of Class~$k$ customers and
the total number of customers present in the network at time $t$. The vector valued process $V(t) = [V_1(t), \dots, V_K(t)]' \in \reals_+^K$ with $V_k(t)$ being the residual
service time for  Class~$k$ customer that is under service. We take $V_k(t) = 0$ whenever $\|Q_k(t)\| = 0$. 
The vector valued process $U(t) = [U_1(t), \ $ $ U^{(e)}_2(t),U^{(ne)}_2(t),$ $\dots,U^{(e)}_L(t),U^{(ne)}_L(t)]' \in \reals_+^{2L-1}$
such that $U_1(t)$ being the remaining time until the next Class~$1$ customer arrival, and at each instant of a Class~$k \geq 2$ customer arrival, 
exponential and non-exponential components of the next interarrival time are generated independently and captured by $U^{(e)}_k$ and $U^{(ne)}_k$, respectively.
Without loss of generality, we can assume that the non-exponential clock $U^{(ne)}_k$ decreases first linearly with rate~1 while exponential clock $U^{(e)}_k$ stays at the same value until
$U^{(ne)}_k$ reaches zero. Thereafter, $U^{(e)}_k$ decreases linearly with rate~1 while other clock stays at zero. 
Next customer arrival happens when both the clocks are zero.
Clearly, at any time~$t$, $U^{(ne)}_k(t) + U^{(e)}_k(t)$ is the remaining time until the next Class~$k$ customer arrival.
We say that Class~$k$ is in exponential phase at time $t$ if $U^{(ne)}_k(t) = 0$. In Section~\ref{RSMCN},
we see that this decomposition of the interarrival times play a crucial role in the construction of regenerations.\\

Let $\mathcal{Y}$ be the \textit{state space} of the process $Y$ and it is adapted to some filtration $\{\mathcal{F}_t~:~t~\geq~0\}$ that is larger or equal to the
natural filtration of the process $Y$. 
Hereafter, we assume that the state space $\mathcal{Y}$ is a complete and separable metric space with the \textit{norm} defined by
$\|y\| = \|q\|+\|u\|+\|v\|, \  y = (q,u,v) \in \mathcal{Y}$, where $\|q\| = \sum_{k=1}^K \|q_k\|$ with $\|q_k\|$ being the number of Class~$k$ customers in the network 
when the state is $y$, $\|u\| = \sum_{k=1}^L|u_k|$ and $\|v\| = \sum_{k=1}^K v_k$. Let $\left\{ P^t\left(y, B \right): B \in \mathcal{B_Y}, t\geq 0 \right\}$ be 
the probability transition kernel of $Y$, where $P^t\left(y, B \right) = P_y \left(Y(t) \in B \right)$ and $P_y$ represents the law of the process when it started 
initially in state $y$.\\

Throughout the paper, we assume that the server at each station is busy whenever there is work to be done (work-conserving) and it stays idle whenever there is no work.
Similar to Proposition 2.1 in \cite{Dai95}, we can establish the strong Markov property of $Y$ for a wide class of queueing disciplines, 
such as FIFO (First-In-First-Out),  LIFO (Last-In-First-Out), priority discipline, processor sharing, etc. 
\section{Regenerative Simulation of Multiclass Networks}
\label{MomRevTimes}
In this section, we identify a sequence of regeneration times and establish required finite moments on the associated regeneration intervals that satisfy the joint CLT.
Furthermore, we propose an alternative regenerative structure that exists when the interarrival times of Class~1 have an exponential or heavier tail distribution.
\subsection{Regenerations}
\label{RSMCN}
For any set $B \in \mathcal{B_Y}$, define the \textit{first hitting time} $\tau_B := \inf\{t \geq 0 : Y(t) \in B\}$, the \textit{first hitting time past $\delta$},
$\tau_B(\delta) := \inf\{t > \delta : Y(t) \in B\}$ and the \textit{first visiting time} $\Gamma_B := \inf\{t > \tau_{B^c} : Y(t) \in B\}$.\\

Let $D := \left\{(\stackrel{\rightarrow}{\mathbf{0}},u, \stackrel{\rightarrow}{\mathbf{0}}) \in \mathcal{Y}:  u_1 < \min_{k \geq 2}\left\{u^{(e)}_k\right\}, u^{(ne)}_l = 0, l\in \{2,\dots,L\}  \right\},$
where $\stackrel{\rightarrow}{\mathbf{0}}$ is all zeros vector of dimension $K$. The set $D$ captures the states that $Y$ can take when the network
is empty and the next transition is triggered by a Class~$1$ customer arrival while all the other non-null exogenous classes are in exponential phase.
Also let $S_{-1}~=~0$ and the $(n+1)^{th}$ revisit instant on set $D$, $S_n = \theta_{S_{n-1}}\circ \Gamma_D$, 
where~$\theta$ is the shift operator on the sample paths of the process $Y$. \\

Let $T_n = S_n + U_1(S_n)$ be the first instant when an arrival of Class~1 finds the system empty past $S_n$.
Just after leaving the set $D$, the state of $Y$ must belong to 
\begin{align*}
\tilde{D}~:=~\left\{(q,u,v) \in \mathcal{Y}:  \|q_1\| = 1, \|q_k\| = 0, \forall \ k \in \{2,\dots,K\}, u^{(ne)}_l = 0, \forall \ l \in \{2,\dots,L\}\right\}, 
\end{align*}
that is, for each $n \geq 0$, $Y(T_n-) \in D$ and $Y(T_n) \in \tilde{D}$. Consider the bounded subsets of $\tilde{D}$ of the form 
$A(r,s)~=~\left\{y \in \tilde{D} :u_1 \leq r_1, \ v_1 \leq s, u^{(e)}_k \leq r_k, \forall\ k \in \{2,\dots,L\}\right\}$, 
for $r = (r_1,\dots,r_L) \in \reals_+^L, s \in \reals_+$, and let $\varphi$ be the probability measure on $\mathcal{B_Y}$ such that 
$\varphi\left(A(r,s)\right) = F_1(r_1)H_1(s) \prod_{k=2}^L\left(1 - e^{-\lambda_k r_k}\right)$,
where $F_k, H_k$ are the distributions of Class~$k$ interarrival and service times (refer to Section \ref{assumptions}). Then it is clear that 
$\varphi\left(\tilde{D}\right) = 1$.\\

One can easily check that when $Y(0) \sim \varphi$, the process $Y$ is a non-delayed regenerative process with 
regeneration times $\left\{ T_n : n\geq 0 \right\}$ and $Y(T_n) \sim \varphi$ for every $n \geq 0$. In Section~\ref{Freq_reg} and Section~\ref{Reg_sim},
we refer to these regeneration times as the \textit{primary} regenerative structure.

\begin{remark}[\textbf{Alternative regenerative structure}]
\normalfont
\label{alt_reg}
 Suppose that Class~1 also satisfies Assumption~(A3) (that is, there exists $f_1 \in \mathcal{H}$ such that 
 $P\left(\xi_{1,1} \in dx \right) \geq \bar{q}_1 f_1(x)dx$ for some $\bar{q}_1 > 0$). Then an exponential component can be extracted from the distribution
 of the interarrival times of Class~1 and the arrival clock process can be written as a vector of exponential and non-exponential components.\\
 
 Let $\left\{\hat{S}_n: n\geq 0\right\}$ be the increasing sequence of times corresponding to customer departure that leaves the network empty when 
 all the interarrival clocks are in exponential phase. 
 Due to the memoryless property of exponential random variables, the underlying Markov process has regenerations at every 
 $\displaystyle \hat{T}_n = \hat{S}_n + \min_{1 \leq k \leq L}\left\{U^{(e)}_k\left(\hat{S}_n\right) \right\}$.
 In Section~\ref{Freq_reg}, we show that when Class~1 has exponential interarrival times, the AVSDE associated with 
 these alternative regenerative structure is smaller than that of previously proposed primary regenerative structure. This is further illustrated with a simple numerical
 example in Section~\ref{Reg_sim}.
\end{remark}

\subsection{Moments of Regeneration Intervals}
\label{moment_tau}
In this section, we establish the finite $p^{th}$ moments of the regeneration intervals,
where $p$ is the parameter used in (A2) that guarantees finite $p^{th}$~moments of the interarrival and service times. 
Hereafter, in addition to (A1) -  (A5), we make the following assumption on the Markov process.
\begin{enumerate}
 \item[(A6)]  There exists $t_0 > 1 $ such that $\displaystyle \lim_{\|y\| \rightarrow \infty} \frac{1}{\|y\|^p} \mathbb{E}_y \left[\|Y(t\|y\|)\|^p\right] = 0, \ \ \forall \ t \geq t_0$.
\end{enumerate}

\begin{remark}
\normalfont
\label{remark:assA6}
\cite{Dai95} shows that, under mild conditions, (A6) holds when Assumptions (A1) and (A2) hold and the fluid model of the network is stable.
 In particular, \cite{Dai95} considers some important networks like re-entrant lines and generalized Jackson networks, and a wide variety of queueing disciplines;
 and shows that the fluid model is stable if the nominal traffic condition (that is, $\rho_i < 1, \text{ for each } i = 1,\dots, d$) holds. 
  Recent work of \cite{SMW12} shows that under certain conditions, fluid model is stable if and only if there exists a Lyapunov function
 (also refer to \cite{SM12}).
\end{remark}

Now we show that the $p^{th}$ moments of the regeneration intervals are finite. Let $C_s := \left\{y \in \mathcal{Y}: \|y\| \leq s\right\}$, for $s \geq 0$.
Lemma \ref{Moment1} below establishes that for any initial state, when the $p^{th}$~moments of the interarrival and service times are finite, 
the time required for the queue length to become smaller than a certain level has finite $p^{th}$~moment. 
Furthermore, it says that there exists a bounded set such that the process visits the set infinitely often and the associated intervals have 
finite $p^{th}$~moments. The proof of Lemma~\ref{Moment1}  mainly depends on (A6).
\begin{lemma}
\label{Moment1}
There exist constants $c_1,\ c_2,\ s_0 > 0$ such that $\mathbb{E}_y\left[(\tau_{C_s}(\delta))^p\right] \leq c_1 + c_2\|y\|^p$, $y~\in~\mathcal{Y}$, $s \geq s_0$, $\delta > 0$.
Furthermore, for any bounded set $A$, $\sup_{y \in A} \mathbb{E}_y\left[(\tau_{C_s}(\delta))^p\right]~<~\infty$, $s \geq s_0.$
\end{lemma}

\begin{lemma}
\label{lemma3}
Suppose that $\tilde{\Gamma}$ is a stopping time with respect to the filtration
$\left\{\mathcal{F}_t:t\geq 0 \right\}$ such that $\displaystyle\inf_{y \in C_s} P_y(\tilde{\Gamma} \leq \delta) > 0$ for some $\delta > 0$ and $s > s_0$.
Then, there exist constants $c_3, c_4$ such that $ \mathbb{E}_y\left[\tilde{\Gamma}^p\right] \leq c_3 + c_4\|y\|^p$ for any $y \in \mathcal{Y}$,
\end{lemma}

Using Lemma~\ref{lemma3}, Lemma~\ref{Moment_D} establishes that whenever the interarrival and service times exhibit finite $p^{th}$ moments, then the $p^{th}$ moment of 
$\Gamma_D~+~U_1(\Gamma_D)$ is finite (recall that $\Gamma_D$ denotes the first visiting time on set $D$ and the process $U_1(\cdot)$ captures 
the remaining time until the next Class~1 customer arrival).

\begin{lemma}
\label{Moment_D}
There exists $\delta > 0$ such that $\inf_{y \in C_s} P_y(\Gamma_D + U_1(\Gamma_D)\leq \delta) > 0$.
Furthermore, there exist constants $c_3,\ c_4 > 0$, $\mathbb{E}_y \left[\left(\Gamma_D + U_1(\Gamma_D)\right) ^p\right] < c_3 + c_4\|y\|^p,\  y \in \mathcal{Y}$.
\end{lemma}

From Lemma~\ref{Moment1} and Lemma~\ref{Moment_D}, we can argue that the first cycle length $T_0$ is a proper random
variable for any initial state $y \in \mathcal{Y}$. Hence, hereafter, we assume that the regenerative process $Y$ is non-delayed (that is, $T_0 = 0$) 
and hence the initial state $Y(0) \sim \varphi$. 
The proposition below establishes finite moments of the regeneration intervals, $\tau_n = T_n - T_{n-1}$, $n~\geq~1$.
The non-lattice property of the distribution of $\tau_1$ simply follows 
from the spreadout Assumption~(A4) on interarrival times of Class~1. From Lemma~\ref{Moment_D},
\begin{align*}
\mathbb{E}_\varphi\left[\tau_1^p \right] &= \mathbb{E}_\varphi \left[\left(\Gamma_D + U_1(\Gamma_D)\right) ^p\right] \\
                                          &< c_3 + c_4\mathbb{E}\left[\|Y(0)\|^p\right] \\
                                          &< \infty.
\end{align*}
Hence, we have the following result.
\begin{proposition}
\label{Reg_MCN}
The distribution of $\tau_1$ is non-lattice and $\mathbb{E}_\varphi[\tau_1^p] < \infty$.
\end{proposition}

\subsection{Moments of $\mathbf{R_1}$}
\label{section:moment_Y}
Let $h$ be a non-negative real valued function defined on $\mathcal{Y}$ and $ R_i = \int_{T_{i-1}}^{T_i} h\left( Y(s) \right)ds$, for $i~\geq~1$.
When $h$ is bounded (e.g., $h(x) = I(x > 10)$), there exists a constant $c$ such that $h < c$ and 
\begin{align*}
\mathbb{E}_\varphi \left[R_1^p \right] = \mathbb{E}_\varphi \left[ \left(\int_0^{\tau_1} h(Y(t))  dt\right)^p \right] \leq c^p\mathbb{E}_\varphi \left[ \tau_1^p  \right] < \infty. 
\end{align*}
But, when $h$ is unbounded, $\mathbb{E}_\varphi \left[R_1^p \right]$ can be infinite even though $\mathbb{E}_\varphi [\tau_1^p] < \infty$. However, we can guarantee these moments even when $h$ is unbounded
under some additional conditions as shown in the following proposition.

For each $r > 0$ and $s \geq s_0$, define
\begin{align*}
 J_{r,s}(y) := \mathbb{E}_y \left[ \left( \int_0^{\tau_{C_s}(\delta)}h(Y(t))  dt \right)^r\right], \ y \in \mathcal{Y},
\end{align*}
where $\delta$ is given by Lemma \ref{Moment_D} and $s_0$ is given by Lemma \ref{Moment1}.
\begin{proposition}
\label{Moment_Y}
Suppose that for a given $r > 0$, there exists $s \geq s_0$ such that $J_{r,s}(\cdot)$
is uniformly bounded on $C_s$ and $\mathbb{E}\left[ J_{r,s}(Y(0)) \right] < \infty$. Then $\mathbb{E}_\varphi \left[R_1^r \right] < \infty$.
\end{proposition}
The following example illustrates one possible application of Proposition~\ref{Moment_Y}.
\begin{example}
\label{Moment_Y_Ex}
\normalfont
Suppose that $h(Y(t)) \leq \|Y(t)\|$, where $\|\cdot\|$ denotes the norm on the state space $\mathcal{Y}$ (refer to Section~\ref{Markov_process}).
For example, if goal is to estimate the steady-state expected number of customers in the network then $h(Y(t)) = \|Q(t)\| \leq \|Y(t)\|$.\\

Now we show that if $p \geq 9$ in (A2) then $\mathbb{E}_\varphi \left[R_1^4 \right] < \infty$ (this is needed for Theorem~\ref{reg_1_thm} 
(iii) to hold). To see this, observe from Lemma~\ref{Moment1} that 
$\mathbb{E}_\varphi\left[ \tau_{C_s}(\delta)^8 \right] \leq c_1 + c_2 \mathbb{E}\left[ \|Y(0)\|^8\right] < \infty$  for any  $s > s_0$.
Similar to Proposition 5.3 of \cite{DM95}, it can be  shown that there exist $s \geq s_0$ and $c < \infty$ such that
 \begin{align*}
 \mathbb{E}_y\left[ \int_0^{\tau_{C_s}(\delta)} \|Y(t)\|^p dt \right] \leq c(\|y\|^{p+1} + 1). 
 \end{align*}
Since $p \geq 9$, we have that $\mathbb{E}\left[\|Y(0)\|^9\right] < \infty$, and hence 
\begin{align*}
\mathbb{E}_\varphi \left[ \int_0^{\tau_{C_s}(\delta)}\|Y(t)\|^8  dt \right] \leq c \left(\mathbb{E}\left[\|Y(0)\|^9\right] + 1\right) < \infty. 
\end{align*}
 Using Cauchy-Schwarz inequality, Jensen's inequality and the fact that $\tau_{C_s}(\delta)~\geq~1$, we can show that
 \begin{align*}
  \mathbb{E}\left[ J_{4,s}(Y(0)) \right]~&\leq~\left(\mathbb{E}_\varphi\left[ \tau_{C_s}(\delta)^8 \right]\mathbb{E}_\varphi  \left[ \int_0^{\tau_{C_s}(\delta)}\|Y(t)\|^8  dt \right]\right)^{\frac{1}{2}} < \infty.
 \end{align*}
 Similarly, uniform boundedness of $J_{4,s}(y)$ on $C_s$ can be established 
and from Proposition \ref{Moment_Y}, it follows that $\mathbb{E}_\varphi \left[R_1^q \right] < \infty$.
\qed
\end{example}

\section{Choice of Optimal Regenerative Structure}
\label{Freq_reg}
In this section, we show that under certain assumptions, if one regenerative structure is a subsequence of another then it is optimal to choose the original
sequence over the subsequence (optimal in the sense that the AVSDE associated with the subsequence is at least as large as that 
associated with the original sequence). Using this result, we show that the selection of interarrival time decomposition with the largest mean exponential component minimizes 
the AVSDE. \\

Suppose that $X$ is a non-delayed regenerative process with the regeneration times $\left\{T_n: n \geq 0\right\}$ and the initial state distribution $\varphi$. 
Let, for each  $i \geq 1$, $\tau_i~=~T_i~-~T_{i-1}$, $R_i = \int_{T_{i-1}}^{T_i} h\left( X(s) \right)ds$, 
where  $h$ is a non-negative, real valued function. Assume that there exists a filtration $\mathcal{G} := \left\{\mathcal{G}_n, n \geq 0\right\}$ and 
a strictly increasing sequence of integer valued stopping times $0~=~\nu_0~<~\nu_1 < \cdots$ adapted to $\mathcal{G}$ with $\nu_n - \nu_{n-1} \stackrel{d}{=} \nu_1, \ n \geq 1$
such that \textbf{(i)}~$\left\{ \left(R_i, \tau_i \right): i \geq 1 \right\}$ is adapted to $\mathcal{G}$, \textbf{(ii)}
$\left\{ \left(R_i, \tau_i \right): i \geq n \right\}$ is independent of $\mathcal{G}_{n-1}$ for all $n \geq 1$ 
and \textbf{(iii)} the sequence $\left\{S_n = T_{\nu_n}: n \geq 0\right\}$ is another regenerative structure of~$X$.
\begin{proposition}
\label{KSvsKT}
Under the above setup, if $\mathbb{E}_\varphi\left[\left(\int_0^{S_1} \left[h\left(X(s)\right) + 1\right]ds\right)^4\right] < \infty$,  
then the AVSDE associated with $\left\{S_n: n \geq 0\right\}$ is at least as large as that 
associated with $\left\{T_n: n \geq 0\right\}$.
\end{proposition}

\cite{ACG95} establishes a similar result on classical regenerative processes and considers stopping times that depend 
only on history of the process, and hence these stopping times turn out to be geometrically distributed 
due to the $i.i.d.$ nature of classical regenerations. Our result is a mild extension of Theorem~1 in \cite{ACG95} in the sense that the stopping time 
can be more general (hence, its distribution may not be geometric).\\ 

Now we consider two important applications of Proposition~\ref{KSvsKT}. 
In both the applications we assume that the function $h$ satisfies $h\left( Y(t) \right) = \hat{h}\left( Q(t), V(t) \right)$ 
(a well known example is when $h(Y(t)) = \|Q(t)\|$ to estimate the steady-state expected customers in the network). 
Then using Proposition~\ref{KSvsKT} and path-wise construction, first we show that the selection of interarrival time decomposition with the largest mean 
exponential component (that is, $\lambda_k = \lambda_{f_k}$) minimizes the AVSDE.
In the second application, we show that when Class~1 interarrival times have an exponential distribution, 
the alternative regenerative structure proposed in Section~\ref{RSMCN} is optimal compared with the primary regenerative structure.
\begin{itemize}
 \item[\textbf{(1)}] 
 Recall that the arrival clock of each Class~$k \in \{2, \dots, L\}$ is updated based on the decomposition~(\ref{TwoMP2}), that is,
\begin{equation*}
        \xi_{k,n} = (1 - \beta_{k,n})\tilde{\xi}_{k,n} + \beta_{k,n} \left(E_{k,n} + Z_{k,n} \right), \ n \in \integers_+,
\end{equation*}
where $\xi_{k,n}$ is the $n^{th}$ interarrival time of class~$k$, $E_{k,n} \sim Exp(\lambda_k)$ for some $\lambda_k \geq \lambda_{f_k}$, 
$\beta_{k,n}$ is a Bernoulli random variable with $P(\beta_{k,n} = 1) = \bar{q}_k$, $Z_{k,n} \sim G^{f_k}_{\lambda_k}$ and 
$\tilde{\xi}_{k,n} \sim \displaystyle \frac{F_k(x) - \bar{q}_k\int_0^x f_k(y)dy}{1- \bar{q}_k}$, $x~\geq~0$. Note that 
$\{ \tilde{\xi}_{k,n} : n\geq~0\}$, $\left\{ E_{k,n} : n\geq 0\right\}$, $\left\{ \beta_{k,n} : n\geq 0\right\}$ and $\left\{Z_{k,n} : n\geq 0\right\}$ 
are $i.i.d.$ sequences and independent of each other.\\

To see the first claim, fix $k \in \{2, \dots, L\}$, and let $\lambda_k^{(1)}, \lambda_k^{(2)} \geq \lambda_{f_k}$ such that 
$\lambda_k^{(1)} < \lambda_k^{(2)}$. When $E_{k,n} \sim Exp\left(\lambda_k^{(1)}\right)$ and 
$E'_{k,n} \sim Exp\left(\lambda_k^{(2)}\right)$, from Theorem~\ref{thm_dec1}, without loss of generality one can write $E_{k,n} = E'_{k,n} + Z'_{k,n},\ n\geq 1$, for an $i.i.d.$ sequence of
positive random variables $\left\{Z'_{k,n},\ n\geq 1\right\}$ independent of $\left\{E'_{k,n},\ n\geq 1\right\}$. 
Suppose that $Y^{(1)}$ (respectively, $Y^{(2)}$) represents the process $Y$ when $\left\{E_{k,n}:n\geq 0 \right\}$ (respectively, $\left\{E'_{k,n}:n\geq 0 \right\}$) is 
the sequence of exponential components and $\left\{Z_{k,n}:n\geq 0 \right\}$ (respectively, $\left\{Z'_{k,n} + Z_{k,n}:n\geq 0 \right\}$)
is the sequence of non-exponential components of interarrival times of class~$k$.\\

Now it is clear that the regenerative structure (denote it by $\{S_n : n \geq 0\}$) associated with $Y^{(2)}$ is a subsequence of the regenerative structure 
(denote it by $\{T_n : n \geq 0\}$) associated with $Y^{(1)}$.
By letting $\tau_i = T_i - T_{i-1}$ and $R_i = \int_{T_{i-1}}^{T_i} h\left( Y^{(1)}(s) \right)ds \left( = \int_{T_{i-1}}^{T_i} h\left( Y^{(2)}(s) \right)ds\right)$  for $i \geq 1$, we can easily see that
$\left\{ \left(R_i, \tau_i \right): i \geq 1 \right\}$ is adapted to 
$\mathcal{G} = \left\{\mathcal{G}_n = \sigma \left( (R_i, \tau_i), \mathbf{1}_i\right) \right\}$ and $\left\{ \left(R_i, \tau_i \right): i \geq n \right\}$ 
is independent of $\mathcal{G}_{n-1}$ for all $n \geq 1$, where $\mathbf{1}_i$ is equal to one if $Y^{(2)}$ is in exponential phase at $T_i$, otherwise, 
it is zero. If we assume that $\mathbb{E}_\varphi\left[\left(\int_0^{S_1} \left[h\left(X(s)\right) + 1\right]\right)^4\right] < \infty$ (for example, if 
$h(Y(t)) \leq \|Y(t)\|$ then, from Example~\ref{Moment_Y_Ex}, this is guaranteed when the $9^{th}$ moments of the interarrival and service times are finite)
then from Proposition~\ref{KSvsKT}, it is optimum to choose $\lambda_k = \lambda_{f_k}$ for each $k = 2,3, \dots, L$.
It is important to note that there exists a stopping time $\nu \geq 1$ adapted to $\mathcal{G}$ such that $S_1 = T_\nu$. In fact, 
$\mathbb{I}(\nu = n) = \mathbb{I}(\mathbf{1}_n = 1)\prod_{i = 0}^{n-1}\mathbb{I}(\mathbf{1}_i = 0)$. Since $\mathbf{1}_i$'s are not mutually independent, 
$\nu$ is not a geometric random variable.\\

\item[\textbf{(2)}] In addition to Assumption (A1) - (A6), assume that Class~1 interarrival times have exponential distribution with rate $\lambda_1$. Recall 
the primary regenerative structure that we proposed earlier: $T_n$ is the $n^{th}$ instant when an arrival of Class~1 finds the system empty 
and all the other non-null exogenous classes are in exponential phase. Now recall the alternative regenerative structure proposed in
Remark~\ref{alt_reg}: $\displaystyle \hat{T}_n = \hat{S}_n + \min_{1 \leq k \leq L}\left\{U^{(e)}_k\left(\hat{S}_n\right)\right\}$, 
where $\hat{S}_n$ is the $n^{th}$ instant when departure of a customer leaves the network empty and all the interarrival clocks are in exponential phase. 
Since the interarrival times of Class~1 are exponentially distributed, $\left\{T_n:n\geq 0\right\}$ is a subsequence of $\left\{\hat{T}_n:n\geq 0\right\}$, and there
exists $\nu \geq 1$ such that $\mathbb{I}(\nu = n) = \mathbb{I}(\mathbf{1}_n = 1)\prod_{i = 0}^{n-1}\mathbb{I}(\mathbf{1}_i = 0)$
and $T_1 = \hat{T}_{\nu}$, where $\mathbf{1}_n$ equals to one if the first arrival immediately after $\hat{S}_n$ is a Class~1 customer, otherwise, it equals to zero;
hence $P\left(\mathbf{1}_n = 1\right) = \frac{\lambda_1}{\lambda_1 + \lambda_2 + \cdots + \lambda_L}$, $n \geq 1$. 
If $ \mathcal{G} = \left\{\mathcal{G}_n = \sigma \left( (R_i, \tau_i), \mathbf{1}_i\right) \right\}$ then from Proposition~\ref{KSvsKT}, 
$\left\{\hat{T}_n:n\geq 0\right\}$ is optimal compared with $\left\{T_n:n\geq 0\right\}$.
\end{itemize}

\section{Simulations}
\label{Reg_sim}
In this section, we present some simple numerical examples to illustrate the regenerative simulation method proposed in the previous sections.
In all the examples, our goal is to  estimate the steady-state expected number of customers in the network shown in Fig.~\ref{MCOQN_Ex}
under the assumption that the queueing discipline is FCFS. 
\begin{figure}[h]
\begin{center}
\includegraphics[height=3.5cm,width=9cm]{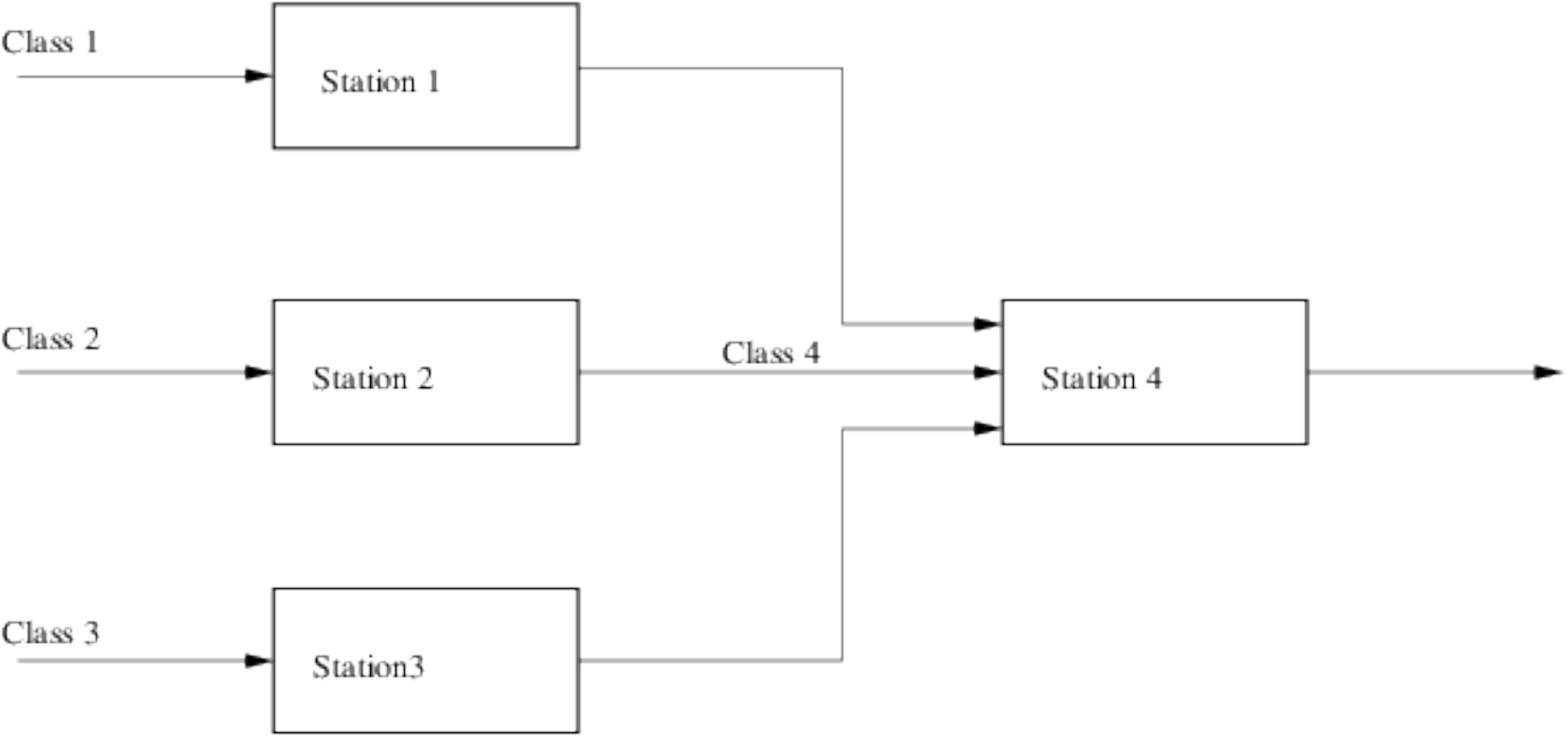}
\caption{\footnotesize  A queueing network with four stations and four classes of customers (three non-null exogenous and one null exogenous).}
\label{MCOQN_Ex}
\end{center}
\end{figure}

In Example~\ref{exp:opt_lambda}, we illustrate that the selection of interarrival time decomposition with the largest mean exponential component minimizes the AVSDE.
Example~\ref{exp:class_1_exp} illustrates the two applications presented in Section~\ref{Freq_reg}. In Example~\ref{exp:pri_better_than_alt} and 
Example~\ref{exp:alt_better_than_pri}, we consider the case where both the primary and alternative regenerative structures exist in the network and neither 
of them is a subsequence of the other. The former example illustrates a scenario where the primary regenerative structure has lower estimated AVSDE compared to the alternative structure, 
while the latter illustrates a scenario where the alternative regenerative structure has lower estimated AVSDE compared to the primary structure. \\

To see how the estimators are constructed in the examples below, it would be useful to recall the regenerative mean and the standard deviation estimators
$\beta(t)= \frac{\sum_{i=1}^{N(t)} R_i}{\sum_{i=1}^{N(t)} \tau_i} \text{  and  }
s(t)~=~\sqrt{\frac{\sum_{i=1}^{N(t)}\left( R_i - \beta(t)\tau_i \right)^2}{\sum_{i=1}^{N(t)} \tau_i}}$, respectively,
where $N(t) = \max\left\{ n \geq 0: T_n \leq t \right\}$, $\tau_i = T_i - T_{i-1}, i \geq 1, $ and $R_i = \int_{T_{i-1}}^{T_i} h\left( X(s) \right)ds, \ i \geq 1$
for a non-negative function $h$. The associated asymptotic $95\%$ confidence interval is $\beta(t) \pm \frac{1.96s(t)}{\sqrt{t}}$.

Also recall that the AVSDE
$\mathcal{K}_{22} = \frac{1}{4\sigma^2}\frac{\mathbb{E}_{\varphi} \left[\left(W_1^2 - \sigma^2 \tau_1 - b W_1\right)^2\right]}{\mathbb{E}_{\varphi}(\tau_1)}$, 
where $W_i = R_i - \bar{r}\tau_i, \ i \geq 0$, $r = \frac{\mathbb{E}_{\varphi} R_1}{\mathbb{E}_{\varphi} \tau_1}$ and
$b = 2\mathbb{E}_\varphi\left( W_1 \tau_1\right)/\mathbb{E}_\varphi \tau_1$. For large $t$ which guarantees small confidence interval,
$\bar{r}$ can be estimated by $\beta(t)$ and $b$ can be estimated by $b(t)= \displaystyle2\frac{\sum_{i=1}^{N(t)} \left(R_i - \beta(t)\tau_i\right)\tau_i}{\sum_{i=1}^{N(t)} \tau_i}$.
Then the AVSDE $\mathcal{K}_{22}$ can be estimated by
\begin{align*}
\mathcal{K}(t) :=\frac{\sum_{i=1}^{N(t)} \left[ \left(\left(R_i - \beta(t)\tau_i\right)^2 - s^2(t)\tau_i - b(t)\left(R_i - \beta(t)\tau_i\right)\right)^2\right]}{4s^2(t)\sum_{i=1}^{N(t)}\tau_i}.
\end{align*}

\begin{example}
\label{exp:opt_lambda}
\normalfont
This example corresponds to network specifications given in Table~\ref{tab:spec1}.

\begin{table}[h]
\centering
\caption{Network specifications}
 \begin{tabular}{|C{0.7cm}|C{2.5cm}|C{3.5cm}|C{0.7cm}|C{1.8cm}|C{1.8cm}|C{1.8cm}|}
       \hline
        Class $k$ &  Interarrival time \newline distribution $F_k$ & Service time \newline distribution $H_k$ & $\lambda_{f_k}$ &
                                     Effective arrival \newline rate $\alpha_k$& Mean service rate $\mu_k$&  Work load at Station $k$ ($\rho_k $)\\
       \hline
        1. &  $Uniform(0,40)$ &  $HyExp(3/4,3/20 ,1/20 )$ & \ \newline  ---&
                                 \ \newline 1/20& \ \newline 1/10& \ \newline  0.5 \\
       \hline
        2. &   $Pareto(10,1/18)$ & $HyExp(1/2,2/3 , 2)$&\ \newline 11/18&
                               \ \newline 2 & \ \newline 1 &  \ \newline 0.5\\
       \hline
        3. &  $Pareto(10,1/9)$&  $Exp(4/3)$ & \ \newline 11/9 & \ \newline 1&
                               \ \newline 4/3 &  \ \newline 0.75\\
       \hline
        4. &  \ \newline --- &  $Exp(5)$ &  \ \newline --- &   \ \newline 3.05 & \ \newline 5 &  \ \newline 0.61\\
       \hline
 \end{tabular}
$Uniform(a,b)$ denotes the uniform distribution over the interval $[a,b]$, $Pareto(a,b)$ represents Pareto distribution with shape parameter $a$
and scale parameter $b$, $HyExp(q,a,b)$ represents hyper-exponential distribution with pdf $g(x) = q a e^{-a x} + (1 - q) b e^{-b x}, \ x \geq 0,$ and $\lambda_{f_k}$ is defined by (\ref{lambda_star}).
\label{tab:spec1}
\end{table}

Observe that Assumptions (A1) - (A5) are trivially satisfied. Since the interarrival time distribution of Class~1 is uniform, 
alternative regenerations are not possible under this set-up (see Remark~\ref{alt_reg}).
This network is a special case of generalized Jackson networks and hence the fluid model is stable as the nominal load $\rho_i < 1, \ i = 1,\dots, 4$ 
(see Remark \ref{remark:assA6}). Thus Assumption (A6) holds and the primary regenerative structure exists. Notice that the $9^{th}$ moments of the interarrival and service times are finite. 
Hence, from Example \ref{Moment_Y_Ex}, $\mathbb{E}_\varphi\left[ R_1^4 + \tau_1^4\right] < \infty$. 
This guarantees asymptotically valid confidence intervals and also the finiteness of the AVSDE.

\begin{table}[h]
\centering
\caption{Simulation results associated with Example~\ref{exp:opt_lambda}}
 \begin{tabular}{|C{5cm}|C{2.9cm}|C{3.2cm}|C{3cm}|}
       \hline
          & $\lambda_k = \lambda_{f_k}$, $k = 2,3$ & $\lambda_k = 1.5\lambda_{f_k}$, $k = 2,3$ & $\lambda_k = 2\lambda_{f_k}$, $k = 2,3$ \\
       \hline
         No. of cycles generated $\left(N(\tilde{T})\right)$ & 13144 & 3892 & 1663 \\
       \hline
        Estimated expected queue length $\left(\beta(\tilde{T})\right)$ &  5.86 & 5.86 & 5.86\\
       \hline
        $95\%$ confidence interval &$5.86 \pm 6.2 \times 10^{-6}$ &$5.86 \pm 6.2 \times 10^{-6}$ & $5.86 \pm 6.3 \times 10^{-6}$\\
       \hline
        Estimated TAVC $\left(s(\tilde{T})^2\right)$ & $1.0 \times 10^3$&$1.0 \times 10^3$ &$1.0 \times 10^3$ \\
       \hline
       Estimated AVSDE
       $\left(\mathcal{K}(\tilde{T})\right)$ &$1.8 \times 10^6$ &$4.6 \times 10^6$ &$6.0 \times 10^6$ \\
       \hline
 \end{tabular}
 
Total duration of the simulation $\tilde{T} = 10^7$ time units.
\label{tab:res1}
\end{table}

Table~\ref{tab:res1} displays the simulation results. It can be observed that the estimated AVSDE is increasing as the rates $\lambda_k, k = 2,3,$ 
of exponential components are increasing and is small when $\lambda_k = \lambda_{f_k}, \ k =2,3$. As expected, the estimated TAVC is not changing with $\lambda_k$. 
\qed
\end{example}

\begin{example}
\normalfont
\label{exp:class_1_exp}
Network specifications are same as those given in Table~\ref{tab:spec1} except that Class~1 interarrival times have exponential distribution with rate 
$\frac{1}{20}$ (hence, there is no change in the effective arrival rate of Class~1). We estimate the required parameters associated with four 
cases shown in Table~\ref{tab:res2}.

\begin{table}[h]
\centering
\caption{Simulation results associated with Example~\ref{exp:class_1_exp}}
 \begin{tabular}{|C{3cm}|C{2.1cm}|C{2.1cm}|C{2.1cm}|C{2.1cm}|C{2.1cm}|}
       \hline
         & Alternative \newline regenerations \newline with \newline $\lambda_k = \lambda_{f_k}$, \newline $k = 2,3$ & Primary \newline regenerations \newline with \newline $\lambda_k = \lambda_{f_k}$, \newline $ k = 2,3$
         & Alternative \newline regenerations \newline with \newline $\lambda_k = 2\lambda_{f_k}$, \newline $k = 2,3$ & Primary \newline regenerations \newline with \newline $\lambda_k = 2\lambda_{f_k}$, \newline $k = 2,3$
         & Alternative \newline regenerations \newline with \newline $\lambda_k = 5\lambda_{f_k}$, \newline $k = 2,3$\\
       \hline
         No. of cycles generated $\left(N(\tilde{T})\right)$ & 393000 & 10226 & 98000 & 1315 &15644\\
       \hline
         Estimated expected queue length $\left(\beta(\tilde{T})\right)$ &  6.08 & 6.08 &  6.08 & 6.08& 6.08\\
       \hline
        $95\%$ confidence interval &$6.08 \pm 7 \times 10^{-6}$ & $6.08 \pm 7 \times 10^{-6}$ & $6.08 \pm 7 \times 10^{-6}$ &$6.08 \pm 7 \times 10^{-6}$&$6.08 \pm 7 \times 10^{-6}$ \\
       \hline
        Estimated TAVC $\left(s(\tilde{T})^2\right)$ & $1.28 \times 10^3$ & $1.28 \times 10^3$ & $1.28 \times 10^3$ & $1.27 \times 10^3$& $1.30 \times 10^3$\\
       \hline
        Estimated AVSDE 
       $\left(\mathcal{K}(\tilde{T})\right)$ &\ \newline $ 1.44\times 10^6$ & \ \newline $ 2.95\times 10^6$ & \ \newline $ 1.45\times 10^6$ & \ \newline $ 1.25\times 10^7$ &\ \newline  $4.2 \times 10^6$\\
       \hline
\end{tabular}

Total duration of the simulation $\tilde{T} = 10^7$ time units.
\label{tab:res2}
\end{table}
This example illustrates three important observations: First, we observe that whether it is the primary or the alternative regenerative
structure, selection of $\lambda_k = \lambda_{f_k}, k = 2,3$ is always an optimal choice. Second, since Class~1 has exponentially distributed interarrival times, 
the alternative regenerative structure associated with $\lambda_k = \lambda_{f_k}, k = 2,3$ is a super-sequence 
of all the other three regenerative structures and as we expect, it has lower estimated AVSDE compared to the others. 
Last, it illustrates that the alternative regenerative structure is optimal compare with the primary structure when $\lambda_k$ is fixed for each Class~$k = 2,3$. 
\qed
\end{example}

\begin{example}
\normalfont
 \label{exp:pri_better_than_alt}
 In this example, as mentioned earlier, we illustrate a case where the primary structure has lower estimated AVSDE compared to the alternative but 
 neither of them is a subsequence of the other. 
 
 We assume that the specifications are same as those given in Table~\ref{tab:spec1} except that Class~1 interarrival times 
 $\xi_{1,i}~\sim~E~+~\bar{W}$, $i \geq 1$, where $E$ is an exponential random variable with rate $10$ and
 $\bar{W}$ is a Weibull random variable with the shape parameter $2$ and the scale parameter $\Gamma(1.5)/19.9$; here $\Gamma(\cdot)$ is the Gamma function. 
 Since the shape parameter is larger than 1, the  distribution  of $\bar{W}$ is superexponential and the distribution of $\xi_{1,1}$ has an exponential tail.
 Notice that even under this setup, the effective arrival rate of Class~1 is $\frac{1}{20}$. 
 \begin{table}[h]
 \centering
\caption{Simulation results associated with Example~\ref{exp:pri_better_than_alt}}
 \begin{tabular}{|C{7cm}|C{3cm}|C{3cm}|}
       \hline
         & Primary \newline regenerations& Alternative \newline regenerations \\
       \hline
         No. of cycles generated $\left(N(\tilde{T})\right)$ & 11630 & 2604 \\
       \hline
         Estimated expected queue length $\left(\beta(\tilde{T})\right)$ &  5.8 & 5.8 \\
       \hline
        $95\%$ confidence  interval &$5.8 \pm 6.0 \times 10^{-6}$ & $5.8 \pm 6.0 \times 10^{-6}$  \\
       \hline
        Estimated TAVC $\left(s(\tilde{T})^2\right)$ & $945$ & $950$ \\
       \hline
        Estimated AVSDE
       $\left(\mathcal{K}(\tilde{T})\right)$ &\ \newline $ 1.47\times 10^6$ & \ \newline $5.24\times 10^6$  \\
       \hline
\end{tabular}

Total duration of the simulation $\tilde{T} = 10^7$ time units.
\label{tab:res3}
\end{table}

Table~\ref{tab:res3} shows the simulation results associated with the case  where $\lambda_k = \lambda_{f_k}$,  $k = 2,3$. 
The primary regenerations are more frequent than the alternative because the ratio of the exponential mean to the total mean 
$\frac{\mathbb{E}\left[E \right]}{\mathbb{E}\left[E + \bar{W}\right]} = 0.005$, and 
the simulation results show that the primary regenerative structure has lower estimated AVSDE compared to the other. 
\qed
\end{example} 
 
\begin{example}
\normalfont
\label{exp:alt_better_than_pri}
 This example provides the results associated with a case where the alternative regenerative structure has lower estimated AVSDE compared to the primary but
 neither of them is a subsequence of the other. 
 
 Network specifications are same as those in Example~\ref{exp:pri_better_than_alt} except that the rate of the exponential random variable $E$ is $0.5$, and 
 the shape and scale parameters of Weibull random variable $\bar{W}$ are $2$ and $\Gamma(1.5)/18$, respectively.
 Notice that the effective arrival rate of Class~1 is again $\frac{1}{20}$. Table~\ref{tab:res4} shows the simulation results associated with the case
 where $\lambda_k = \lambda_{f_k}$, $k = 2,3$.
 \begin{table}[h]
 \centering
\caption{Simulation results associated with Example~\ref{exp:alt_better_than_pri}}
 \begin{tabular}{|C{7cm}|C{3cm}|C{3cm}|}
       \hline
         & Primary \newline regenerations& Alternative \newline regenerations \\
       \hline
         No. of cycles generated $\left(N(\tilde{T})\right)$ & 2841 & 53224 \\
       \hline
         Estimated expected queue length $\left(\beta(\tilde{T})\right)$ &  5.8 & 5.8 \\
       \hline
        $95\%$ confidence interval &$5.8 \pm 6.3 \times 10^{-6}$ & $5.8 \pm 6.3 \times 10^{-6}$  \\
       \hline
        Estimated TAVC $\left(s(\tilde{T})^2\right)$ & $990$ & $995$ \\
       \hline
       Estimated AVSDE
       $\left(\mathcal{K}(\tilde{T})\right)$ &\ \newline $ 4.50\times 10^6$ & \ \newline $1.08\times 10^6$  \\
       \hline
\end{tabular}

Total duration of the simulation $\tilde{T} = 10^7$ time units.
\label{tab:res4}
\end{table}

In the present scenario, $\frac{\mathbb{E}\left[E \right]}{\mathbb{E}\left[E + \bar{W}\right]} = 0.1$.
The simulation results show that the alternative regenerations are more frequent than the primary regenerations and the former has lower estimated AVSDE compared to the latter. 
\qed
\end{example}

\section{Summary}
In this paper, first we showed that, under mild conditions, a random variable with an exponential or heavier tail can be re-expressed as a mixture of sums of independent
random variables where one of the constituents is exponential distributed. 
This result was illustrated for several distributions, such as, hyper-exponential, Gamma, log-normal, Pareto, Weibull, etc.
Using this, we developed two implementable regenerative simulation methods for multiclass 
open queueing networks where the interarrival times have exponential or heavier tails. 
The first method is applicable even when one of the non-null exogenous classes has interarrival times with superexponential distributions.
We referred to the sequence of regeneration times associated with this method as the primary regenerative structure.
The second method is applicable only when all the interarrival time distributions of non-null exogenous classes have exponential or heavier tails, and the 
associated regeneration times are known as the alternative regenerative structure.

Under mild stability conditions, we established that the finite $p^{th}$ moments of the interarrival and service times are sufficient to guarantee 
the finite $p^{th}$ moments of the regeneration cycle lengths $\tau_i, i~\geq~1$. Furthermore, we studied some sufficient conditions that 
guarantee finite $4^{th}$ moments of $R_i, i~\geq~1$. These moments are crucial as they guarantee asymptotically valid confidence intervals
for steady-state performance measures of interest. 
Finally, under certain assumptions, we showed that when one regenerative structure is a subsequence of another,
the AVSDE associated with the subsequence is at least as large as that associated with the original sequence.
One application of this result is to show that the selection of interarrival time decomposition with the largest mean exponential component minimizes the AVSDE. 
In another application, we showed that when at least one of the non-null exogenous classes has exponentially distributed interarrival times, the AVSDE
associated with the alternative regenerative structure is smaller than that associated with the primary regenerative structure.

In \cite{SS13}, we briefly discuss a different regenerative simulation technique that can be 
applicable even when the interarrival times of some (or all) non-null exogenous classes have superexponential tails.
This may be a fruitful direction for further research.

\section{Proofs}
\label{sec:proofs}

\begin{proof}[Proof of Theorem \ref{thm_dec1}]
First notice that
\begin{align}
\label{G_lam_eq2}
G_\lambda^f(x)  &= \frac{f(a)}{\lambda} + \int_{a}^x(f(y) + \frac{1}{\lambda}f^{'}(y))dy,  \text{  for   }x \geq a.
\end{align}
Suppose that $\lambda \geq \lambda_f$. Then from (\ref{lambda_star}), $f(y) + \frac{1}{\lambda}f^{'}(y) \geq 0$ for every $y \in \reals$.
Since $\lambda_f$ is finite, it is not difficult to show that 
$\displaystyle \lim_{x \rightarrow -\infty} f(x) = \lim_{x \rightarrow \infty} f(x) = 0$.  Thus $\int_{a}^{\infty} f'(y)dy~=~-f(a)$ and it follows that,
\begin{eqnarray*}
   \int_{a}^{\infty}(f(y) + \frac{1}{\lambda}f^{'}(y))dy + \frac{f(a)}{\lambda} = 1.
\end{eqnarray*}
Therefore, from (\ref{G_lam_eq2}), $G^f_\lambda$ is a probability distribution and it is clear that $G_\lambda^f$ has a point mass $\frac{f(a)}{\lambda}$
at $a$ and the density $f(\cdot) + \frac{1}{\lambda}f^{'}(\cdot)$ on $(a,\infty)$.

To see (\ref{EZ_dec}), observe that when a random variable $Z$ is independent of $E$, the characteristic function of $E+Z$
\begin{eqnarray}
\label{ap1}
 \phi_{E+Z}(t) = \phi_Z(t)\phi_E(t)= \phi_Z(t)\left(\frac{\lambda}{\lambda - it}\right), \ t \in \mathbb{R}.
\end{eqnarray}
Suppose that $Z \sim G^f_\lambda$. Then 
\begin{eqnarray}
\label{ap2}
 \phi_Z(t) &=& \frac{f(a)}{\lambda}e^{iat} + \int_a^{\infty}e^{izt}\left(f(z) + \frac{1}{\lambda}f^{'}(z)\right)dz
\cr                 &=& \frac{(\lambda - it)}{\lambda}\left[\int_a^{\infty}e^{izt}f(z)dz  \right].
\end{eqnarray}
By substituting (\ref{ap2}) in (\ref{ap1}), we have $ \phi_{E+Z}(t) =  \phi_\xi(t)$.
Since characteristic function uniquely identifies the probability distribution, it follows that (\ref{EZ_dec}) holds. 

To prove the converse, notice that when $E$ and $Z$ are independent,
 $\phi_Z(t)~=~\frac{(\lambda - it)}{\lambda} \phi_\xi(t)$. When $\phi_\xi$ is absolutely integrable, 
 from inversion formula (e.g, Theorem 12.3 in Chapter II of \cite{SHI96}),
 $f(x) = \frac{1}{2\pi}\int_{-\infty}^\infty e^{-ixt} \phi_\xi(t)dt$ and
\begin{align*}
 P(Z \leq x) - P(Z \leq y) &= \lim_{c\rightarrow \infty} \frac{1}{2\pi}\int_{-c}^c \frac{e^{-ity} - e^{-itx}}{it}\phi_Z(t)dt
\end{align*}
for every $x$ and $y$ $(x > y)$ at which $P(Z \leq \cdot)$ is continuous. Then
\begin{align*}
 P(Z \leq x) - P(Z \leq y) &= F(x) - F(y) - \lim_{c\rightarrow \infty} \frac{1}{2\pi \lambda}\int_{-c}^c \left(e^{-ity} - e^{-itx} \right)\phi_\xi(t)dt \\
             &= F(x) - F(y) - \frac{\left(f(y) - f(x) \right)}{\lambda} \\
\end{align*}
Let $y \rightarrow -\infty$ on both sides then it is clear that 
the probability distribution of $Z$ should be of the form (\ref{G_lam_eq}). However, under continuity of $f'$, it is easy to observe that $G^f_\lambda$ is a 
valid distribution only if $\lambda \geq \lambda_f$. 
\end{proof}
\ \\
\begin{proof}[Proof of Lemma \ref{Moment1}]
 From Assumption (A6), it follows that for any given $\delta~>~0$ there exist $t_0 > 1$ and $s_0 > \frac{\delta}{t_0}$ such that
 $\mathbb{E}_y\left[ \|Y(t_0\|y\|)\|^p\right] \leq \frac{1}{2}\|y\|^p$ for all $y$ with $\|y\|~>~s_0$.
That means, for every $y \in \mathcal{Y}$ and for any fixed $s \geq s_0$,
\begin{eqnarray}
\label{bound1}
\mathbb{E}_y\left[ \|Y(t_0\|y\|)\|^p\mathbb{I}(y \notin C_s)\right] &\leq& \frac{1}{2}\|y\|^p.
\end{eqnarray}
Let $t(y) := \left\{
           \begin{array}{l l}
                    \delta+1 & \quad \text{if $y \in C_s$},\\
                    t_0\|y\| & \quad \text{if $y \notin C_s$}\
           \end{array} \right.
$
and define the sequence of stopping times as follows:
$\zeta_0 = 0$ and $\zeta_{i+1} = \zeta_i + \theta_{\zeta_i} \circ t\left(Y(\zeta_i) \right)$ for every $ i \geq 0$.
Let $\hat{Y}_i = Y(\zeta_i), i\geq 0$ and $\hat{N}~=~\inf\left\{i \geq 0: \hat{Y}_i \in C_s\right\}$.

Then $\{\hat{Y}_i\}_{i\geq 0}$ is a Markov chain with transition probability kernel defined by $\hat{P}(y,A) := P_y(Y(t(y)) \in A)$ 
and $\hat{N}$ is the first hitting time on set $C_s $ along the Markov chain $\hat{Y}$.
Let $\hat{\mathcal{F}}_n := \sigma(\hat{Y}_0, \hat{Y}_1, \cdots , \hat{Y}_n)$, $n = 0,1, \dots $ be the natural filtration generated by $\hat{Y}$.
Then (\ref{bound1}) can be expressed as
\begin{eqnarray}
\label{bound2}
\mathbb{E}_y\left[ \|\hat{Y}_1\|^p\mathbb{I}(y \notin C_s )\right] &\leq& \frac{1}{2}\|y\|^p, \text{  for all } y \in \mathcal{Y}
\end{eqnarray}
Since $\tau_{C_s}(\delta) \leq \zeta_{\hat{N}}$, it is enough to show that $\mathbb{E}_y(\zeta_{\hat{N}}^p) < \infty$.
Notice that,
\begin{eqnarray*}
	\zeta_{\hat{N}} = \sum_{i=0}^{\hat{N} - 1} t(\hat{Y}_i) = t_0\|\hat{Y}_0\|\mathbb{I}(Y(0) \notin C_s) + (\delta+1)\mathbb{I}(Y(0) \in C_s) + \sum_{i=1}^{\infty} t_0\|\hat{Y}_i\|\mathbb{I}(i \leq \hat{N} - 1),
\end{eqnarray*}
and
\begin{eqnarray*}
	\mathbb{E}_y\left[\|\hat{Y}_i\|^p\mathbb{I}(i \leq \hat{N} - 1)\right] &\leq& \mathbb{E}_y\left[\|\hat{Y}_i\|^p\mathbb{I}(i \leq \hat{N} )\right]\\
\cr                            &=& \mathbb{E}_y\left[\|\hat{Y}_i\|^p\mathbb{I}(\min(\|\hat{Y}_1\|, \cdots, \|\hat{Y}_{i-1}\|) > s)\right] \\
\cr                            &=& \mathbb{E}_y\left[\mathbb{I}(\min(\|\hat{Y}_1\|, \cdots, \|\hat{Y}_{i-2}\|) > s)\mathbb{E}_y\left[\|\hat{Y}_i\|^p\mathbb{I}(\|\hat{Y}_{i-1}\| \notin C_s)|\hat{\mathcal{F}}_{i-1}\right]\right].
\end{eqnarray*}
Under strong Markov property of $Y$ and (\ref{bound2}), we have
\begin{align*}
 &\mathbb{E}_y\left[\|\hat{Y}_i\|^p\mathbb{I}(\|\hat{Y}_{i-1}\| \notin C_s)|\hat{\mathcal{F}}_{i-1}\right] \leq \frac{1}{2}\|\hat{Y}_{i-1}\|^p \text{   and}\\
 &\mathbb{E}_y\left[\|\hat{Y}_i\|^p\mathbb{I}(i \leq \hat{N} - 1)\right] \leq
                          \frac{1}{2}\mathbb{E}_y\left[\|\hat{Y}_{i-1}\|^p\mathbb{I}(\min(\|\hat{Y}_1\|, \cdots, \|\hat{Y}_{i-2}\|) > s)\right].
\end{align*}
Now using recursion, we have $\mathbb{E}_y\left[\|\hat{Y}_i\|^p\mathbb{I}(i \leq \hat{N} - 1)\right] \leq \frac{1}{2^i}\|y\|^p < \infty, \ \ \ \forall \ i = 1,2, \dots$
Recursive application of Minkowski's inequality imply that for every $y \in \mathcal{Y}$
\begin{eqnarray*}
	\left(\mathbb{E}_y\left[\zeta_{\hat{N}}^p\right]\right)^{\frac{1}{p}} &\leq&  \max (\delta+1,t_0\|y\|) + t_0 \sum_{i=1}^{\infty} \left(\mathbb{E}_y\left[\|\hat{Y}_i\|\mathbb{I}(i \leq \hat{N} - 1)\right]^p\right)^{\frac{1}{p}}\\
\cr                                                                      &\leq&    \max (\delta+1,t_0\|y\|) + t_0 \sum_{i=1}^{\infty} \frac{1}{\left(2^\frac{1}{p}\right)^i} \|y\| \\
\cr                                                                      &\leq&    \delta+1 + t_0\left(1 - \frac{1}{2^\frac{1}{p}}\right)^{-1}\|y\|.
\end{eqnarray*}
Hence
\begin{eqnarray*}
	\mathbb{E}_y\left[\zeta_{\hat{N}}^p\right]  \leq  2^{p-1}\left((\delta+1)^p + \left[t_0\left(1 - \frac{1}{2^\frac{1}{p}}\right)^{-1} \|y\|\right]^p\right) = c_1 + c_2\|y\|^p,
\end{eqnarray*}
where $c_1 = 2^{p-1}(\delta+1)^p$ and $c_2 = 2^{p-1}\left[t_0 \left(1 - \frac{1}{2^\frac{1}{p}}\right)^{-1}\right]^p$.
Therefore, for any bounded set $A$ we have	$\sup_{y \in A}\mathbb{E}_y\left[\tau_{C_s}(\delta)^p\right] < \infty$.
\end{proof}
\ \\
\begin{proof}[Proof of Lemma \ref{lemma3}]
     From Lemma \ref{Moment1},  $\mathbb{E}_y\left[\tau_{C_s}(\delta)^p \right] \leq c_1 + c_2\|y\|^p$ for some constants $c_1$ and $c_2$.
     Now define $\zeta_0 := 0$ and $\zeta_{i+1} := \zeta_{i} + \theta_{\zeta_i} \circ \tau_{C_s}(\delta), \ i \geq 0$.
     Let $\hat{N} = \min\left\{i \geq 1 : \left(\theta_{\zeta_i} \circ \tilde{\Gamma}\right) \leq \delta\right\}$,
     then it is clear that
     \begin{align*}
	     \tilde{\Gamma} &\leq \zeta_{\hat{N}} + \delta \\
	              &=    \zeta_{\hat{N} -1} + \theta_{\zeta_{\hat{N}-1}} \circ \tau_{C_s}(\delta) + \delta \\
		      &=    \sum_{i=0}^{\hat{N} - 1}\theta_{\zeta_{i}} \circ \tau_{C_s}(\delta) + \delta \\
		      &=    \sum_{i=0}^{\infty}\left(\theta_{\zeta_{i}} \circ \tau_{C_s}(\delta)\right)\mathbb{I}\left(i < \hat{N}\right) + \delta.
     \end{align*}
     Using Minkowski's inequality, for any $y \in \mathcal{Y}$,
     \begin{align}
	  \label{Min_1}
	     \left(\mathbb{E}_y \left[\tilde{\Gamma}^p\right]\right)^{\frac{1}{p}} &\leq \left(\mathbb{E}_y\left[\tau_{C_s}(\delta)^p\right]\right)^{\frac{1}{p}}
	       + \sum_{i=1}^{\infty}\left(\mathbb{E}_y\left[\left(\theta_{\zeta_{i}} \circ \tau_{C_s}(\delta)\right)^p\mathbb{I}\left(i < \hat{N}\right)\right]\right)^{\frac{1}{p}} + \delta.
     \end{align}
     Now consider
     \begin{align}
       \mathbb{E}_y\left[\left(\theta_{\zeta_{i}} \circ \tau_{C_s}(\delta)\right)^p\mathbb{I}\left(i < \hat{N}\right)\right]
           &\leq \mathbb{E}_y\left[\left(\theta_{\zeta_{i}} \circ \tau_{C_s}(\delta)\right)^p\mathbb{I}\left(i \leq \hat{N}\right)\right] \nonumber\\
	   &=\mathbb{E}_y\left[\mathbb{I}\left(i \leq \hat{N}\right)\mathbb{E}_y\left[\left(\theta_{\zeta_{i}} \circ \tau_{C_s}(\delta)\right)^p|\mathcal{F}_{\zeta_i}\right]\right]\nonumber\\
	   &=\mathbb{E}_y\left[\mathbb{I}\left(i \leq \hat{N}\right)\mathbb{E}_{Y(\zeta_i)}\left[(\tau_{C_s}(\delta))^p\right]\right]\nonumber\\
	   &\leq M P_y\left(i \leq \hat{N}\right)\label{M_ind},
     \end{align}
     where $M = \sup_{y \in C_s} \mathbb{E}_y\left[(\tau_{C_s}(\delta))^p\right]$, which is finite from Lemma \ref{Moment1}.
     
     Since $\mathbb{I}\left(i \leq \hat{N}\right) = \mathbb{I}\left(\left(\theta_{\zeta_j} \circ \tilde{\Gamma}\right) > \delta, \forall j = 1,\dots,i-1\right)$ and
     $Y(\zeta_i) \in C_s,\ i \geq 1$, we have
     \begin{align*}
	     P_y\left(i \leq \hat{N}\right) &= \mathbb{E}_y\left[ \mathbb{E}_y\left[\mathbb{I}\left(i-1 < \hat{N}\right)|\mathcal{F}_{\zeta_{i-1}}\right]\right] \nonumber\\
	                                   &= \mathbb{E}_y\left[\mathbb{I}\left(i-2 < \hat{N}\right) P_{Y(\zeta_{i-1})}\left(\tilde{\Gamma} > \delta\right)\right] \nonumber\\
					   &\leq (1-\epsilon) P_y\left(i-2 < \hat{N}\right), 
     \end{align*}
     where $\displaystyle\epsilon := \inf_{y \in C_s} P_y(\tilde{\Gamma} \leq \delta) > 0$. Using recursion,
     \begin{align}
     \label{eqn:P_y}
	     P_y\left(i \leq \hat{N}\right) \leq (1-\epsilon)^{i-1}.
     \end{align}
     Hence from (\ref{Min_1}), (\ref{M_ind}) and Lemma \ref{Moment1},
     \begin{align*}
	     \left(\mathbb{E}_y \left[\tilde{\Gamma}^p\right]\right)^{\frac{1}{p}} &\leq (c_1 + c_2\|y\|^p)^\frac{1}{p}
	                            + M^\frac{1}{p}\sum_{i=1}^{\infty}\left((1-\epsilon)^{\frac{1}{p}}\right)^{i-1} + \delta = (c_1 + c_2\|y\|^p)^\frac{1}{p} + c,	
     \end{align*}
     where $c = M^\frac{1}{p}\frac{1}{1 - (1-\epsilon)^{\frac{1}{p}}} + \delta $. Thus $\mathbb{E}_y \left[\tilde{\Gamma}^p\right] \leq 2^{p-1}\left(c_1 + c_2\|y\|^p + c^p\right)$.
     We can finish the proof by taking $c_3 = 2^{p-1}(c_1 +c^p)$ and $c_4 = 2^{p-1} c_2$.
\end{proof}
\ \\
\begin{proof}[Proof of Lemma \ref{Moment_D}]
 From Lemma \ref{lemma3}, now it is sufficient to show that $\displaystyle\inf_{y \in C_s} P_y( \Gamma_D + U_1(\Gamma_D) \leq \delta) > 0$ for some $\delta > 0$ and $s > s_0$.
  This can be proved by finding times $0 < t_1 < t_2 < \infty$ such that, with positive probability, by time $t_1$ exactly one customer per each Class 
  $k \in \{2,\dots,L\}$ arrives into the network and by time $t_2$  system can hit the set $D$.

  Fix any $s \geq s_0$ and choose $t_1 = s$ then by time $t_1$ each Class $k \in \{2,\dots,L\}$ can enter into exponential phase immediately after the first arrival
  with probability $\bar{q}_k$ for any initial state $Y(0) \in C_s$.

 Since the network is open (in fact $(I - P')^{-1}$ exists), there is a positive probability of any customer leaving the network by making at most
 $K$ distinct transitions, where the transition refers to customers' class change upon completion of their service. Under (A5), there exists $x'>0$ such that 
 $P(\xi_{1,1} \geq x \text{ and } \sum_{l=1}^K \eta_{l,1} \leq x - x') > 0$. That is, the system can become empty,
 with a positive probability, within a finite time, call it $t'$, when there are no more arrivals from any Class $k \in \{2, \dots,L\}$. We can choose $t_2 > t'$
 such that there is an arrival of Class~1 into empty system within $t_2$ when all the other non-null exogenous classes are in exponential phase.

 By letting $\delta = \max\{t_2,1\}$, we can conclude that $\inf_{y \in C_s} P_y(\Gamma_D + U_1(\Gamma_D)\leq \delta) > 0.$
\end{proof}
\ \\
\begin{proof}[Proof of Proposition \ref{Moment_Y}]
     Recall the sequence of stopping times defined in Lemma \ref{lemma3}, that is, $\zeta_0 = 0 $ and
     $\zeta_{i+1} = \zeta_{i} + \theta_{\zeta_i} \circ \tau_{C_s}(\delta), \ i \in \integers_+$.     
     Let $\hat{N}~=~\min\left\{i \geq 1 : \left(\theta_{\zeta_i} \circ \left(\Gamma_D + U_1(\Gamma_D)\right)\right) \leq \delta\right\}$,
     then it is clear that $T_1 = \Gamma_D + U_1(\Gamma_D) \leq \zeta_{\hat{N}+1}$ and, since $h$ is a positive function,
     \begin{align*}
	     \int_0^{T_1} h\left( Y(t)\right) dt  &\leq  \sum_{i=0}^{\hat{N}} \int_{\zeta_i}^{\zeta_{i+1}} h\left( Y(t) \right)  dt.
     \end{align*}
     Using Minkowski's inequality, 
     \begin{align*}
           \left(\mathbb{E}_\varphi\left[\left( \int_0^{T_1} h\left( Y(t) \right)dt\right)^r\right]\right)^{\frac{1}{r}} &\leq  \sum_{i=0}^{\infty} \left(\mathbb{E}_\varphi \left[ \left(\int_{\zeta_i}^{\zeta_{i+1}} h\left( Y(t) \right) dt \right)^r \mathbb{I} \left( i \leq \hat{N} \right)\right]\right)^{\frac{1}{r}}\\
           &= \left(\mathbb{E}_\varphi\left[ \left(\int_0^{\zeta_1} h\left( Y(t) \right) dt\right)^r\right]\right)^{\frac{1}{r}} \\
             & \ \ \ \ \ + \sum_{i=1}^\infty  \left(\mathbb{E}_\varphi\left[ \left(\int_{\zeta_i}^{\zeta_{i+1}} h\left( Y(t) \right) dt\right)^r \mathbb{I}\left( i \leq \hat{N} \right)\right]\right)^{\frac{1}{r}}.
     \end{align*}
     Observe that
     \begin{align*}
      \mathbb{E}_\varphi\left[ \left(\int_0^{\zeta_1} h\left( Y(t) \right) dt\right)^r \right]
                   &= \mathbb{E}_\varphi\left[ \left(\int_0^{\tau_{C_s}(\delta)} h\left( Y(t) \right) dt\right)^r \right] \\
                   &= \mathbb{E}\left[ J_{r,s}\left( Y(0)\right) \right]  < \infty.
     \end{align*}
     Now consider,
     \begin{align*}
       \mathbb{E}_\varphi\left[\left(\int_{\zeta_i}^{\zeta_{i+1}} h\left( Y(t) \right) dt \right)^r \mathbb{I}\left( i \leq \hat{N} \right)\right]
	&\leq \mathbb{E}_\varphi\left[\mathbb{I}\left(i \leq \hat{N}\right)\mathbb{E}_\varphi\left[ \left.\left(\int_{\zeta_i}^{\zeta_{i+1}} h\left( Y(t) \right) dt\right)^r\right|\mathcal{F}_{\zeta_i}\right]\right]\nonumber\\
	&= \mathbb{E}_\varphi\left[\mathbb{I}\left(i \leq \hat{N}\right)\mathbb{E}_{Y(\zeta_i)}\left[ \left(\int_0^{\tau_{C_s}(\delta)} h\left( Y(t) \right) dt\right)^r\right]\right]\nonumber \\
	&= \mathbb{E}_\varphi\left[\mathbb{I}\left(i \leq \hat{N}\right) J_{r,s}\left(Y(\zeta_i)\right) \right] \nonumber.
     \end{align*}
     Since $Y(Y(\zeta_i)) \in C_s, i \geq 1$,
     \begin{align*}
        \mathbb{E}_\varphi\left[\left(\int_{\zeta_i}^{\zeta_{i+1}} h\left( Y(t) \right) dt \right)^r \mathbb{I}\left( i \leq \hat{N} \right)\right]
        &\leq M P_\varphi\left(i \leq \hat{N}\right),
     \end{align*}
     where $M = \sup_{y \in C_s} J_{r,s}(y)$ and it is finite from the given hypothesis.

     If we replace $\tilde{\Gamma}$ with $\Gamma_D + U_1(\Gamma_D)$ in Lemma~\ref{lemma3} then from (\ref{eqn:P_y}), 
     we have $P_\varphi \left(i \leq \hat{N}\right) \leq (1-\epsilon)^{i-1}$, where 
     $\displaystyle\epsilon~:=~\inf_{y \in C_s} P_y(\Gamma_D + U_1(\Gamma_D) \leq \delta)$ and it is positive from Lemma~\ref{Moment_D}. Hence
     \begin{align*}
	\left(\mathbb{E}_\varphi\left[ \left(\int_0^{T_1} h\left( Y(t) \right)dt\right)^r\right]\right)^{\frac{1}{r}}
	                                               \leq \left(\mathbb{E}\left[ J_{r,s}\left( Y(0)\right)\right]\right)^{\frac{1}{r}}
	                                               + M^\frac{1}{r}\sum_{i=1}^{\infty}\left((1-\epsilon)^{\frac{1}{r}}\right)^{i-1} 
	                                                < \infty.
     \end{align*}
     This establishes $\mathbb{E}_\varphi \left[R_1^r \right] < \infty$.
\end{proof}
\ \\

\begin{proof}[Proof of Proposition~\ref{KSvsKT}]
Let $W_i = R_i - r\tau_i$, where $\bar{r} = \frac{\mathbb{E}_\varphi\left[R_1\right]}{\mathbb{E}_\varphi\left[\tau_1\right]}$. 
From Wald's identity, $\mathbb{E}_\varphi\left[\tilde{W}^2_{\nu_1}\right] = \mathbb{E}_\varphi\left( \nu_1 \right) \mathbb{E}_\varphi\left[W^2_1 \right]$
and $\mathbb{E}_\varphi\left[ S_1 \right] = \mathbb{E}_\varphi\left( \nu_1 \right) \mathbb{E}_\varphi\left[\tau_1 \right]$, 
where $\tilde{W}_n = \sum_{i=0}^n W_i$, $n \geq 0$ with $W_0 = 0$. 
Hence, $\sigma^2 = \frac{\mathbb{E}_\varphi\left[W^2_1\right]}{\mathbb{E}_\varphi\left[\tau_1\right]} = 
\frac{\mathbb{E}_\varphi\left[\tilde{W}^2_{\nu_1}\right]}{\mathbb{E}_\varphi\left[S_1\right]}$.
From Theorem \ref{reg_1_thm}, the AVSDEs with respect to the regenerative structures $\left\{T_n: n \geq 0\right\}$
and $\left\{S_n: n \geq 0\right\}$ are  
$\mathcal{K}^T_{22}~=~\frac{1}{4\sigma^2}\frac{\mathbb{E}_{\varphi} \left[\left(W_1^2 - \sigma^2 \tau_1 - b_T W_1\right)^2\right]}{\mathbb{E}_{\varphi}(\tau_1)}$ and
$\mathcal{K}^S_{22}~=~\frac{1}{4\sigma^2}\frac{\mathbb{E}_{\varphi} \left[\left(\tilde{W}_{\nu_1}^2 - \sigma^2 S_1 - b_S \tilde{W}_{\nu_1}\right)^2\right]}{\mathbb{E}_{\varphi}(S_1)}$,
respectively, where $b_T = 2\frac{\mathbb{E}_{\varphi}\left[W_1 \tau_1\right]}{\mathbb{E}_{\varphi}\left[\tau_1\right]}$
and $b_S~=~2\frac{\mathbb{E}_{\varphi}\left[\tilde{W}_{\nu_1} S_1\right]}{\mathbb{E}_{\varphi}\left[S_1\right]}$.

It is now enough to prove that
\begin{align*}
\mathbb{E}_{\varphi} \left[\left(\tilde{W}_{\nu_1}^2 - \sigma^2 S_1 - b_S \tilde{W}_{\nu_1}\right)^2\right] \geq
\mathbb{E}_{\varphi}(\nu_1)\mathbb{E}_{\varphi} \left[\left(W_1^2 - \sigma^2 \tau_1 - b_T W_1\right)^2\right].
\end{align*}
Let $\nu = \nu_1$ and $M_n = \tilde{W}_n^2 - \sigma^2 T_n - b_S \tilde{W}_n$. 
Then clearly, $\left\{M_n: n \geq 0\right\}$ is a martingale with respect to $\mathcal{G}$, and 
$\mathbb{E}_\varphi\left[M^2_n\right] < \infty, \ n \geq 0$ 
Then using the fact $\nu$ is a stopping time with respect to $\mathcal{G}$, we have
 $\mathbb{E}_\varphi\left[ M_\nu^2\right]~=~\mathbb{E}_\varphi\left[\sum_{n=1}^\nu \left(M_n - M_{n-1} \right)^2\right]\label{Corr_eq1}$.
 Furthermore,
 \begin{align*}
\mathbb{E}_\varphi\left[ \tilde{W}_\nu S_1 \right] &= \mathbb{E}_\varphi\left[ \sum_{i=0}^\nu W_i \sum_{j=0}^\nu \tau_j \right] \\
                                                   &= \mathbb{E}_\varphi\left[ \sum_{i=0}^\nu W_i\tau_i \right]
                                                        + \mathbb{E}_\varphi\left[ \sum_{0 \leq i < j \leq \nu} W_i \tau_j \right]
                                                        + \mathbb{E}_\varphi\left[ \sum_{0 \leq j < i \leq \nu} W_i \tau_j \right]\\
                                                   &= \mathbb{E}_\varphi\left( \nu \right) \mathbb{E}_\varphi\left[W_1 \tau_1 \right]
                                                        + \mathbb{E}_\varphi\left[\tau_1\right]\mathbb{E}_\varphi\left[\sum_{j=1}^\nu \sum_{i = 0}^{j-1} W_i \right]
                                                        + \mathbb{E}_\varphi\left[W_1\right]\mathbb{E}_\varphi\left[\sum_{i=1}^\nu \sum_{j = 0}^{i-1} \tau_j \right]\\
                                                   &= \mathbb{E}_\varphi\left( \nu \right) \mathbb{E}_\varphi\left[W_1 \tau_1 \right] +
                                                       \mathbb{E}_\varphi\left[\tau_1\right]\mathbb{E}_\varphi\left[\sum_{j=1}^\nu \tilde{W}_{j-1} \right].
\end{align*}
Then it is clear that $b_S = b_T  + 2 \frac{C}{\mathbb{E}_\varphi\left( \nu \right)}$,
where $\displaystyle C = \mathbb{E}_\varphi\left[\sum_{n=1}^\nu \tilde{W}_{n-1} \right]$. Now we can write
\begin{align*}
 M_n = M_{n-1} + \left(W_n^2 - \sigma^2 \tau_n - b_T W_n \right) + 2 W_n \left(\tilde{W}_{n-1} - \frac{C}{\mathbb{E}_\varphi\left( \nu \right)} \right)
 \end{align*}
 and
 \begin{align*}
 \left(M_n - M_{n-1}\right)^2 &= \left(W_n^2 - \sigma^2 \tau_n - b_T W_n \right)^2 + 4 W^2_n \left(\tilde{W}_{n-1} - \frac{C}{\mathbb{E}_\varphi\left( \nu \right)} \right)^2 \\
                              &\ \ \ \ + 4 \left(W_n^3 - \sigma^2 \tau_n W_n - b_T W^2_n \right)\left(\tilde{W}_{n-1} - \frac{C}{\mathbb{E}_\varphi\left( \nu \right)} \right).
\end{align*}
Then we have
\begin{align*}
 \mathbb{E}_\varphi\left[ M_\nu^2\right]
             &=\mathbb{E}_\varphi\left[\sum_{n=1}^\nu \left(M_n - M_{n-1} \right)^2\right] \\
             &= \sum_{n=1}^\infty \mathbb{E}_\varphi\left[ \left(M_n - M_{n-1} \right)^2\mathbb{I}(\nu \geq n)\right] \\
             &= \mathbb{E}_\varphi\left[ \left(W_1^2 - \sigma^2 \tau_1 - b_T W_1 \right)^2\right]\mathbb{E}_\varphi (\nu) \nonumber\\
             &\ \  \ \ + 4 \mathbb{E}_\varphi\left(W^2_1\right) \left\{\mathbb{E}_\varphi\left[\sum_{n=1}^\nu \tilde{W}_{n-1}^2\right] + \frac{C^2}{\left(\mathbb{E}_\varphi (\nu)\right)^2}\mathbb{E}_\varphi (\nu)\right. \nonumber\\
             &\ \  \ \ - \left.2 \frac{C}{\mathbb{E}_\varphi (\nu)} \mathbb{E}_\varphi\left[\sum_{n=1}^\nu \tilde{W}_{n-1}\right]\right\}\nonumber\\
             &\ \  \ \ + 4 \mathbb{E}_\varphi\left[W_1^3 - \sigma^2 \tau_1 W_1 - b_T W^2_1 \right]  \left(\mathbb{E}_\varphi\left[\sum_{n=1}^\nu \tilde{W}_{n-1}\right] - C \right). 
\end{align*}
Cancellation of some terms will result in
\begin{align*}
 \mathbb{E}_\varphi\left[ M_\nu^2\right]
             &= \mathbb{E}_\varphi\left[ \left(W_1^2 - \sigma^2 \tau_1 - b_T W_1 \right)^2\right]\mathbb{E}_\varphi (\nu) + 4 \mathbb{E}_\varphi\left(W^2_1\right) \left(\mathbb{E}_\varphi\left[\sum_{n=1}^\nu \tilde{W}_{n-1}^2\right]
                                          - \frac{C^2}{\mathbb{E}_\varphi (\nu)} \right).
\end{align*}
To prove the proposition, now it is enough to show that
$\displaystyle\mathbb{E}_\varphi\left[\sum_{n=1}^\nu \tilde{W}_{n-1}^2\right] \geq \frac{C^2}{\mathbb{E}_\varphi (\nu)}$. 
From Cauchy-Schwarz inequality, $ \mathbb{E}_\varphi\left[\tilde{W}_{n-1}^2\mathbb{I}(\nu \geq n)\right]P_\varphi(\nu \geq n) \geq \left(\mathbb{E}_\varphi\left[\tilde{W}_{n-1}\mathbb{I}(\nu \geq n)\right]\right)^2$.
Hence,
\begin{align*}
\mathbb{E}_\varphi\left[\sum_{n=1}^\nu \tilde{W}_{n-1}^2\right]
                      &=\sum_{n=1}^\infty \mathbb{E}_\varphi\left[\tilde{W}_{n-1}^2\mathbb{I}(\nu \geq n)\right]\\
                      &\geq \sum_{n=1}^\infty \frac{\left(\mathbb{E}_\varphi\left[\tilde{W}_{n-1}\mathbb{I}(\nu \geq n)\right]\right)^2}{P_\varphi(\nu \geq n)}\\
                      &= \mathbb{E}_\varphi (\nu)\sum_{n=1}^\infty \left[\left(\frac{\mathbb{E}_\varphi\left[\tilde{W}_{n-1}\mathbb{I}(\nu \geq n)\right]}{P_\varphi(\nu \geq n)}\right)^2\frac{P_\varphi(\nu \geq n)}{\mathbb{E}_\varphi (\nu)}\right]\\
                      &\geq \mathbb{E}_\varphi (\nu)\left(\sum_{n=1}^\infty \frac{\mathbb{E}_\varphi\left[\tilde{W}_{n-1}\mathbb{I}(\nu \geq n)\right]}{\mathbb{E}_\varphi (\nu)}\right)^2,
\end{align*}
where the last inequality follows from Jensen's inequality because $q_n := \frac{P_\varphi(\nu \geq n)}{\mathbb{E}_\varphi (\nu)},\ n\geq 0$ is a probability
distribution on non-negative integers. Since $R_n \geq 0, n~\geq~0$, it is easy to observe that  
\begin{align*}
 \sum_{n=1}^\infty \mathbb{E}_\varphi\left[\tilde{W}_{n-1}\mathbb{I}(\nu \geq n)\right] = C.
\end{align*}
\end{proof}
\bibliographystyle{acm}
\bibliography{references}
\end{document}